\documentclass[final,3p,times,12pt]{elsarticle}
\usepackage{graphicx,amssymb, verbatim,enumerate,amsmath}
\usepackage{amsthm}
\usepackage{upgreek}

\usepackage{mathrsfs}
\usepackage{mathabx}

\usepackage[pdfpagelabels,plainpages=false]{hyperref}
\hypersetup{ linkcolor =blue,
 citecolor = cyan,
 urlcolor = blue,
colorlinks = true, bookmarksopen = true, bookmarksnumbered = true,
bookmarksopenlevel = {1}, linktocpage = true,
pdfstartview = FitH,
pdftitle = {Nonlocal time porous medium equation with fractional time derivative},
pdfsubject
= {article},
pdfauthor = {Djida et al.},
pdfcreator = {pdftextify}, pdfproducer = {pdftextify}
}

\newcommand\R{\mathbb{R}}
\newcommand{\e}{\upvarepsilon}
\newcommand\K{\mathcal{K}}
\newcommand{\Norm}[1]{\left\lVert#1\right\rVert}
\newcommand\norm{H^{\frac{\sigma}{2}}(\mathbb{R}^N)}
\newcommand{\fcaputoa}{{}^{c}_{a} D^{\upgamma}_{t}}
\newcommand{\fcaputo}{{}^{c}_{0} D^{\upgamma}_{t}}
\def\LL{\mathrm{L}}

\newtheorem{theorem}{Theorem}
\newtheorem{corollary}[theorem]{Corollary}
\newtheorem{lemma}[theorem]{Lemma}

\newtheorem{definition}[theorem]{Definition}

\newtheorem{remark}[theorem]{Remark}

\journal{}
\begin{document}

\begin{frontmatter}
\title{Nonlocal time porous medium equation with fractional time derivative}

\author[label1,label2]{Jean-Daniel~Djida}
\ead{jeandaniel.djida@usc.es}

\author[label1]{Juan J.~Nieto}
\ead{juanjose.nieto.roig@usc.es}

\author[label3]{I.~Area\corref{cor1}}
\ead{area@uvigo.es}
\cortext[cor1]{Corresponding author}

\address[label1]{Departamento de Estat\'{\i}stica, An{\'a}lise Matem\'{a}tica e Optimizaci\'on, Universidade de Santiago de Compostela, 15782 Santiago de Compostela, Spain}
\address[label2]{African Institute for Mathematical Sciences, AIMS-Cameroon, P.O. Box 608, Limb\'e Crystal Gardens, South West Region, Cameroon}
\address[label3]{Departamento de Matem\'atica Aplicada II,  E.E. Aeron\'autica e do Espazo,  Universidade de Vigo,  Campus As Lagoas s/n, 32004 Vigo, Spain}

\begin{abstract}
We consider nonlinear nonlocal diffusive evolution equations, governed by fractional Laplace-type operators, fractional time derivative and involving porous medium type nonlinearities. Existence and uniqueness of weak solutions are established using approximating solutions and the theory of maximal monotone operators. Using the De Giorgi-Nash-Moser technique, we prove that the solutions are bounded and H\"{o}lder continuous for all positive time.
\end{abstract}


\end{frontmatter}

\section{Introduction}
In this paper we analyze nonlocal nonlinear equations with fractional time derivative. The basic operators involved are the so-called fractional Laplacian in $\R^N$  \cite{Hitch2012,Vazquez2011NonlinearDW}, $(-\Delta)^{s}$ with $s\in(0,1)$ and the so-called fractional derivative in the sense of extended Caputo or Marchaud  $\partial_{t}^{\upgamma}$, $\upgamma \in(0,1)$ in $(0,T)$, for $T>0$ \cite{samko1993fractional,Marchaud,Ferrari2017}. \medskip

Nonlinear diffusion problems of the parabolic type involving fractional Laplacian and fractional time derivative operators have recently attracted the interest of many authors (see e.g.~\cite{Vazquez2011NonlinearDW, Allen1, Bucur} and references therein). In this paper we focus our attention on a problem posed on the whole space domain. More precisely, we consider the fractional diffusion equation with fractional time derivative known as the fractional time porous medium equation
\begin{equation}\label{Eq:main_problem}
\partial_t^{\upgamma} w + \mathcal{L}^{s} \left(w^{m} \right) = f,  \quad  \text{for all }~ (t,x) \in Q,\quad ~ w(x,0) = g(x), ~\text{for all} ~ x \in \R^{N},
\end{equation}
where $m>1$, $s,\upgamma \in(0,1)$, and $N \geq 1$. The linear operator $ \mathcal{L}^{s}$ is a fractional power of the Laplacian subject to suitable Dirichlet boundary conditions. The initial value data $g \in L^{1}(\R^N)$ has compact support and the right hand side $f \in L^{\infty}(Q)$ is a smooth bounded forcing function.\medskip

The evolutionary nonlinear equation \eqref{Eq:main_problem} with the fractional time derivative is analogous to the abstract evolutionary equations with usual time derivative
\begin{equation}\label{abstractevolution}
\partial_t w + \mathcal{L} \upvarphi(w) = f, 
\end{equation}
where $\upvarphi : \R \to \R$ is a nondecreasing function. This type of equation arises in flow in porous medium or plasma physics, depending on the choice of $\upvarphi$. Solution to \eqref{abstractevolution} has been studied \cite{Crandall1987} for the regularizing effect of \eqref{abstractevolution}. In the case $ \mathcal{L}^{s}:=\left(-\Delta \right)^{s}$ and $\upvarphi(w) := w^{m}$, with $m>1$, and $f=0$, the problem becomes the nonlocal porous medium equation. Regularity and existence results for this specific case have been obtained in \cite{Pablo1, Pablo2, Pablo3, RosOton}, by using De Giorgi methods in order to show that the solutions of~\eqref{abstractevolution} are $s$-H\"{o}lder continuous, which implies H\"{o}lder continuity. The control of the oscillation in the nonlocal setting followed the procedure developed in \cite{CaffarelliChanVasseur} for a linear problem with a rough kernel, combined with some ideas which deal with the nonlinearity borrowed from \cite{AthanasopoulosCaffarelli}.\medskip

Recently another variant of the problem~\eqref{abstractevolution} was extended in~\cite{ Allen1, Allen2, Allen}, in the case $ \mathcal{L}^{s}:=\left(-\Delta \right)^{s}$ and $\upvarphi(w):= w$, with fractional derivative $\partial_{t}^{\gamma}$ in the sense of extended Caputo, so that the problem becomes a nonlocal linear problem
\begin{equation}
\label{nolocalproblem}
\partial_{t}^{\gamma}w - \int \left[w(t,y)-w(t,x) \right]\K(t,x,y) = f(t,x).
\end{equation}
Again, the De Giorgi techniques were used to prove H\"{o}lder continuity for solutions to \eqref{nolocalproblem} of divergence form which is a linear analogue of our problem \eqref{Eq:main_problem}.\medskip

The authors in \cite{Allen2} utilized solutions that are weak-in-time to prove H\"{o}lder continuity. The same authors also adapted the methods in \cite{Allen3} to prove H\"{o}lder continuity for a nonlocal porous medium equation with inverse potential pressure. As already mentioned in \cite{Allen4}, considering weak-in-time solutions is advantageous for existence and regularity results.\medskip

As far as we know, the case of fully nonlinear and nonlocal variant of \eqref{abstractevolution} is still open. More precisely, the case $ \mathcal{L}^{s}:=\left(-\Delta \right)^{s}$, $\upvarphi(w) := w^{m}$, with $m>1$ and the fractional derivative $\partial_{t}^{\gamma}$ in the sense of extended Caputo or Marchaud. This is the main aim of this paper., i.e to analyze problem~\eqref{Eq:main_problem}. \medskip

The paper is organized as follows. In Section~\ref{sec:2} the main results are stated, regarding existence, uniqueness, and regularity. Later, in order to prove existence and uniqueness, a preliminary step will be to show that the approximate solutions are strong in time. So, in Section \ref{sec:functional_setting}, we give some properties of the fractional time derivative and its discretized version which will allow us to use Steklov averages, hence strong solution in time. Next we provide some functional inequalities related to the fractional Laplacian. We will end the section with the concept of weak solutions. In Section~\ref{sec:Existence_Uniqueness} we provide proofs of existence and uniqueness of solutions stated in Theorem~\ref{theo:existence}. Section~\ref{sec:regularity} is devoted for the proof of H\"{o}lder regularity solution for the problem \eqref{Eq:main_problem}.\medskip

We recall the notation that will be intensively used throughout the paper:\medskip
\begin{itemize}
\item $\gamma$ denotes the order the extended Caputo derivative or Marchaud derivative;

\item $s$  denotes the order of the spatial fractional operator associated to the fractional Laplacian;

\item $a$ stands for the initial time for which our equation is defined;

\item $N$ refers to the space dimension;

\item $\Lambda$ denotes the elliptic positive constant which gives the bound of the kernel of the fractional Laplacian;

\item $\Lambda_{1}, \Lambda_{2}$ stands for the elliptic positive constant which gives the bound of the kernel of the fractional derivative;

\item $\upvarepsilon$ refers to the time length of the discrete approximation;

\item $t, \tau$ denote time variables;

\item $\psi$ stands for a cut-off function;

\item $Q=(0,T) \times \R^{N}$ denotes the domain;

\item $\Gamma_{R} = \{ x \in B_{R}, t \in (0,T)\}$ denotes the space with center the ball with radius $B_{R}$;

\item $w(t,x)$ denotes signed solution to problem \eqref{Eq:main_problem};

\item Finally, for the sake of brevity, $\upvarphi(w) = \left|w\right|^{m-1}w = w^m$, $m>1$.\medskip
\end{itemize}

\section{Nonlocal evolution problem, and main results}\label{sec:2}

\subsection{Nonlocal evolution problem. Existence and uniqueness of solutions}\medskip

Our first goal is to prove an existence and uniqueness result. Actually, since our result holds for more general nonlinearities than the powers, we will adopt the more general context and consider the problem
\begin{equation}\label{Eq:TFPME.Problem}
\begin{cases}
\displaystyle \partial_{t}^{\upgamma} + \mathcal{L}^{s}(\varphi(w))= f, &  ~ {\rm in} \quad Q,\\
\displaystyle w(0,x)=w_0(x), & ~{\rm in} \quad \R^N,
\end{cases}
\end{equation}
where $\upvarphi:\R\to\R$ is a continuous, smooth and increasing function such that $\upvarphi'>0$, $\upvarphi(\pm\infty)=\pm\infty$ and $\upvarphi(0)=0$ (see e.g. \cite{AthanasopoulosCaffarelli}, where this type of conditions were considered). The leading example will be $\upvarphi(w)=|w|^{m-1}w$ with $m>0$\,. \medskip

\begin{definition}\label{def.H*sols}
Let $w \in L^{2}\left([0,T];\mathcal{D}(\mathcal{B}) \right) \cap H^{\frac{\upgamma}{2}} \left([0,T]; H^*(\R^N)\right)$. Then, $w$ is a solution to \eqref{Eq:main_problem} if $\varphi(w) \in L^1([0,T],  H^{*}(\R^N) )$ and
\begin{equation}\label{Eq:testH*}
\int_0^T \int_\R^N \psi \partial_{t}^{\upgamma} w = \int_{0}^{T}\mathcal{E}(w,\psi)= \int_{0}^{T}\int_{\R^{N}} f ~\psi, \,\qquad\forall\, \psi \in C^1_c([0,T], H^{*}(\Omega)),
\end{equation}
providing that $\mathcal{L}^{s}$ is an isomorphism from $H$ into $H^*$ as introduced in \cite{BonforteSireVazquez}.
\end{definition}

The main result on the existence and uniqueness is as follows.    
\begin{theorem}\label{theo:existence}
Let $w(x,0) \in H^{*}(\R^N)$. Then for all $t>0$, $ m > m^{*}:= (\upgamma N - 2s)/\upgamma N $, and $N>2s/\upgamma$, there exists a unique solution
$\displaystyle w \in L^{2}\left([0,T];\mathcal{D}(\mathcal{B}) \right) \cap H^{\frac{\upgamma}{2}} \left([0,T]; H^*(\R^N)\right)$ to problem \eqref{Eq:TFPME.Problem}, satisfying the contraction property
$
\Norm{w(t)- v(t)}_{H^{*}} \leq \Norm{w(0)- v(0)}_{H^{*}}.
$ 
\end{theorem} 

 
The proof of this result included in Section \ref{sec:Existence_Uniqueness} is essentially based on the approximating solution and the maximal monotonicity operators through the Crandall-Liggett theory. The key point is based on showing that the fractional derivative and the fractional Laplacian operators can be suitably defined as a maximal monotone operator on the Hilbert spaces $L^{2}\left([0,T];\mathcal{D}(\mathcal{B}) \right)$ and $H^*$, respectively. We follow the  approach in \cite{BonforteSireVazquez} dealing with the fractional Laplacian case, which happens to be the extension of the Brezis approach \cite{Brezis} where the operator ${\mathcal{L}}^{s}$ is characterized as the sub-differential of a convex functional. Regarding the fractional derivative, we stress out that the maximal monotonicity has been study in e.g. \cite{Ito,Bajlekova}. The method produces not only existence and uniqueness of a semigroup of solutions in the linear and nonlinear setting, but also a number of important estimates, typical of evolution processes governed by maximal monotone operators. \medskip

\subsection{Regularity of solutions}
In order to study the H\"{o}lder regularity, we rewrite the problem \eqref{Eq:main_problem} in the form 
\begin{equation}
\label{Eq:main-vartheta}
\partial_{t}^{\upgamma}\vartheta(w) - \int \left[ w(t,y) - w(t,x) \right] \K(t,x,y) \ dy = f(t,x),
\end{equation}
where $\vartheta := \vartheta (w)= \upvarphi^{-1}(w)$ is a continuous increasing real-valued function, with $\vartheta (\pm \infty) = \pm \infty$ satisfying
\begin{equation}\label{Eq:properties_beta}
\begin{cases}
\displaystyle \vartheta^{\prime}(\tau)~ \text{exists for all}~\tau \neq 0,\\
\displaystyle \vartheta^{\prime}(0) \geq c_{1} \geq 0, ~ \text{for some constant}~~ c_{1},\\
\displaystyle \partial_{\tau}^{\upgamma}\vartheta(\tau) \geq \frac{1}{(1-\upgamma)\Gamma(1-\upgamma)} t^{1-\upgamma}, \quad \text{for}~~\tau > 0,\\
\displaystyle \vartheta(0) = 0. 
\end{cases}
\end{equation}

\noindent Problems \eqref{Eq:main-vartheta} and \eqref{Eq:properties_beta} can be thought as a boundary version of a singular equation which includes the porous media equation. In addition to the existence and uniqueness result that will be proved for this problem, we also prove that $w$ is a continuous function of $x$ and $t$ up to the boundary of $Q$. Its modulus of continuity will depend on $s,\upgamma$, and $w$. If in addition, we assume that $w$ has near zero a homogeneous behavior such as that of the porous media, i.e., $w^{m}$, $m > 1$, then we obtain a H\"{o}lder modulus of continuity. \medskip

In view of the form of equation ~\eqref{Eq:main-vartheta}, the equivalent weak formulation is needed. To this end, we formally multiply~\eqref{Eq:main-vartheta} by a test function $\psi \in C^{1}_{0}\left(\R^N\right) \times C^{1}(0,T)$, such that 

\begin{equation}\label{weak-nonlocal}
\int_{0}^{T}\int_{\R^{N}} \psi~ \partial_{t}^{\gamma} \vartheta (w)  dx \ dt +  \int_{0}^{T}\mathcal{E}(w,\psi)= \int_{0}^{T}\int_{\R^{N}} f ~\psi dx \ dt.
\end{equation}
Formally for $w=\left(\mathrm{Tr}(\vartheta)\right)^{\frac{1}{m}}$ and under the assumption that $f \in L^{\infty}((a,T)\times \R^{N})$ we have the following definition.

\begin{definition}\label{Conceptweaksolutions}
\noindent We say that the pair $(w,\vartheta)$ is a weak solution to problem~\eqref{Eq:main-vartheta} if for the function $w$ we have $w \in L^{1} \left((a,T); W^{1,1}_{loc}(\R^N)\right)$, $w = \left(\mathrm{Tr}(\vartheta)\right)^{\frac{1}{m}} \in L^{1} \left((a,T) \times \R^{N} \right)$, and if for all $T>0$  and $f \in L^{\infty}((a,T) \times \R^{N})$, the following relation holds
\begin{multline}\label{Eq:weak}
-\int_{\R^{N}} \int_{a}^{T} \left( \nabla w(t,x), \nabla \psi(t,x) \right) dx~dt  -\int_{\R^{N}}\int_{a}^{T} \vartheta (t,x)\partial^{\gamma}_{t}\psi(t,x)dt~dx \\
 + \gamma \int_{\R^{N}} \int_{a}^{T}\int_{a}^{t} \left[\vartheta(t,x)-\vartheta(\tau,x) \right]\left[\psi(t,x)- \psi(\tau,x) \right] \left(t-\tau\right)^{-1-\gamma} d\tau~dt~dx \\
-\int_{\R^{N}} \int_{a}^{T} \left[\psi(t,x)\vartheta(a,x)+ \psi(a,x)\vartheta(t,x) \right]\left(t-a\right)^{-\gamma}~dt~dx \\
 + \int_{\R^{N}}\int_{a}^{T}\vartheta(t,x)\psi(t,x) \left[ \frac{1}{(T-t)^{\gamma}} + \frac{1}{(t-a)^{\gamma}} \right] d\tau~dt~dx 
 =\int_{a}^{T}\int_{\R^N} f(t,x)\psi(t,x)dx~dt.
\end{multline}
\end{definition}

As mentioned previously a powerful method of construction of solutions of evolutions equations is the so-called implicit time discretization \cite{Crandall-Liggett}. The construction of an approximate solution of the problem in a time interval $(a,T)$ proceeds by dividing the time interval $(a,T)$ into $k$ subintervals of length $\varepsilon = T/k$ and then defining the approximate solution $(w_{\varepsilon}, \vartheta_{\varepsilon})$ to \eqref{Eq:main-vartheta} constant on each subinterval in the following way: 
\begin{equation}\label{Eq:main-2}
\upgamma \varepsilon \sum _{i<j} \frac{\left[\vartheta(a + \varepsilon j, x)- \vartheta(a + \varepsilon i, x) \right]}{\left( \varepsilon(j-i)  \right)^{1+\upgamma}} 
= \int_{\R^N} \left[ w(a + \varepsilon j, y) - w(a + \varepsilon j, x) \right] \K(t,x,y) \ dy + f(t,x),
\end{equation}
for each $-\infty <i < j \leq k$, with $w_{\varepsilon}(0,x) = g$ on $\R^N$. In each step, for $i,j \in {\mathbb{N}}$, $w(a+\varepsilon i, x) = \left( \mathrm{Tr}(\vartheta)(\varepsilon(a+\varepsilon i, x ))\right)^{1/m}$ is known and $w(a+\varepsilon j, x)$ and $\vartheta(a+\varepsilon j, x)$ are the unknowns. \medskip

Next we state our main theorem on H\"{o}lder regularity that will be proved in Section~\ref{sec:regularity}. The key ingredient of the proof is the De Giorgi's method, (see~\cite{Allen1, Pablo3, CaffarelliChanVasseur,  Allen2, DeGiorgi}). We control the oscillation following the procedure developed in \cite{Allen1, Pablo3, Allen2}, combined with some ideas to deal with the nonlinearity that will appear in the fractional derivative. These techniques will be accompanied by the help of Crandall-Liggett theorem~\cite{ Crandall-Liggett} and some ideas from \cite{AthanasopoulosCaffarelli}. More precisely, we will prove that the oscillation of the solution in space-time $\beta$-cylinders of radius $R$, is reduced in a fraction of the cylinder $
\mathcal{C}_{\gamma R}$, $\zeta<1$, at least by a constant factor $\kappa_*$. This implies $\beta$--H\"older continuity.\medskip
\begin{theorem}\label{theo:regularity}
Let $w$ be a bounded weak solution which satisfies~\eqref{Eq:weak} with $f \in L^{\infty}(Q)$ and $\vartheta$ satisfying also~\eqref{Eq:properties_beta}. Furthermore, let the kernels $\mathscr{K}$ and $\K$ satisfy respectively~~\eqref{eq:kernel_derivative} and~\eqref{Eq:kernel_laplacian}. Then $w$ is H\"{o}lder continuous in $Q$ for some exponent $\beta\in(0,1)$ such that $w\in C^{\beta}(Q)$, with $\beta = \log\kappa_*/\log \zeta$.
\end{theorem}

\section{Preliminary results}\label{sec:functional_setting}

In this section we recall some previous results as well as we prove some new results that will be useful in our further analysis.\medskip

\subsection{The operators}
When the permeability of the medium changes over time such as in porous medium equation, it might be interesting to use a fractional time derivative. Among the different fractional derivatives existing in the literature, in this paper we consider the extended Caputo or Marchaud derivative, since the problem we study is in the divergence from. The usual Caputo derivative for $0<\upgamma<1$ is defined by
\[
\fcaputoa v(t):= \frac{1}{\Gamma(1-\upgamma)} \int_{a}^{t}\left(t-\tau\right)^{-\upgamma}v^{\prime}(\tau)d\tau.
\]
Using an integration by parts and proceeding as defined in \cite{Ferrari2017, Allen1, Bucur, Allen2, BucurFerrari2016}, defining $v(t) \equiv v(a)$ for $t<a$, the extended form or the Marchaud derivative is defined as
\begin{equation}\label{Eq:rescaled_caputo}
\partial_t^{\upgamma} v(t,\cdot)= \upgamma \int_{-\infty}^t \left[ v(t,\cdot)-v(\tau,\cdot)\right] \mathscr{K}(t,\tau) \ d\tau.
\end{equation}

The kernel $\mathscr{K}$ also satisfies the conditions { }
\begin{equation}\label{eq:kernel_derivative}
\mathscr{K}(t,t-\tau) = \mathscr{K}(t+\tau, t) \qquad\textrm{ and } \qquad \frac{\Lambda_{1}}{(t-\tau)^{1+\upgamma}} \leq \mathscr{K}(t,\tau) \leq \frac{\Lambda_{2}}{(t-\tau)^{1+\upgamma}}.
\end{equation}

The formulation~\eqref{Eq:rescaled_caputo} is also known as the Marchaud derivative \cite{samko1993fractional,Marchaud,Stinga1}. The reason of working with formulation \eqref{Eq:rescaled_caputo} is that it allows one to easily utilize the nonlocal nature of the fractional time derivative for regularity purposes. This was succesfully accomplished for divergence problems in \cite{Allen1} as well as for nondivergence problems in \cite{Allen2, Allen}.\medskip

\begin{lemma}\cite[Lemma 2.2]{Allen4}\label{lem:switch}
If $\partial_t^{\upgamma} \eta \in L^2(-\infty,T; L^{2}(\R^N)) $,  and $\eta(t)=0$ for $t<-M$ for some $M>0$, 
           then $\partial_t^{-1} \partial_t^{\upgamma} \eta = \partial_t^{\upgamma} \partial_t^{-1} \eta$ in $L^2(-\infty,T; L^{2}(\R^N))$ where 
            \[
             \partial_t^{-1} \eta := \int_{-\infty}^t \eta(\tau) \ d \tau. 
            \]
\end{lemma}

We consider the usual Sobolev space \cite{ Hitch2012, VasquezCaffarelli} 
\[
H^s(\R):=\left\{h:L^{2}(\R): \int_{\R}(1+|\xi|^{2s})|\mathcal{F}h(\xi)|^{2} < +\infty \right\},
\]
endowed with the norm
\[
\|h\|_{H^s(\R)}:=\|h\|_{L^2(\R)} + \int_{\R} |\xi|^{2s}|\mathcal{F}h(\xi)|^{2}.
\]

The nonlocal operator in space $\mathcal{L}^{s}$ is the so-called fractional Laplacian operator with measurable kernel $\K$ which is defined by (see \cite{Hitch2012, Pablo2, Pablo3, VasquezCaffarelli, CaffarelliCabre1997, Fausto} for more details)
\begin{equation}\label{eq:operatorls}
\mathcal{L}^{s}h(x):=\left(-\Delta \right)^{s} h(x)= C_{s} \, \text{P.V.} \int_{\R} \left[ h(x)- h(x-y)\right] \K(x,y) \, dy = \mathcal{F}^{-1}(|\xi|^{2s}(\mathcal{F}h)),
\end{equation}
for $h \in H^s(\R)$, $x\in \R$, where $C_{s}=\pi^{-(2s+1/2)}\Gamma(1/2+s)/\Gamma(-s)$, and P.V. stands for the Cauchy principal value.
The kernel $\K$ is symmetric and satisfies
\begin{equation}\label{Eq:kernel_laplacian}
\K(\cdot,x,y)=\K(\cdot,y,x) \quad\textrm{and} \quad\chi_{\{ |x-y|\leq 3 \}} \frac{\Lambda^{-1/2}}{|x-y|^{1+2s}} \leq \K(\cdot,x,y) \leq \frac{\Lambda^{1/2}}{|x-y|^{1+2s}},
\end{equation}
for some elliptic constant $\Lambda>0$.\medskip

The bounds in \eqref{Eq:kernel_laplacian} imply that the kernel is oscillating and irregular, or referred as rough kernel \cite{Pablo3}. \medskip

Let $ \mathcal{L}^{s}$ defined in the domain  $\mathcal{D}( \mathcal{L}^{s} )\subset\LL^2(\R^N)\to \LL^2(\R^N)$ be a  positive self-adjoint operator. As it has been shown in \cite{Hitch2012, BonforteSireVazquez} the operator  $\mathcal{L}^{s}$ enjoys a property of having a discrete spectrum and a $\LL^2(\R^N)$ ortonormal basis of eigenfunctions. If we denote by $\lambda_k$  its eigenvalues written in increasing order and repeated according to their multiplicity, and by  $\{\Phi_k\}$ the corresponding $\LL^2(\R^N)$-normalized eigenfunctions; then it turns out that they form an orthonormal basis for $\LL^2(\R^N)$\, for all eigenvalues strictly positive. In this way, we can associate to $ \mathcal{L}^{s}$ the bilinear form in order to introduce the action of the operator $\mathcal{L}^{s}$ defined in \eqref{eq:operatorls} in a weak sense. We can associate a  bilinear form \cite{Pablo1, CaffarelliChanVasseur, DianaFelixVasquez2014, schilling2007}
\begin{equation*}
\label{quadratic-form}\mathcal{E}(f,g)=\frac12\int_{\mathbb{R}^{N}}%
\int_{\mathbb{R}^{N}}(f(x)-f(y))(g(x)-g(y))\K(x,y)\,dxdy,
\end{equation*}
and the quadratic form $\overline{\mathcal{E}}(f)=\mathcal{E}(f,f)$.   For kernels satisfying the symmetry condition~\eqref{Eq:kernel_laplacian} and functions $f,\, g\in C^2_0(\mathbb{R}^N)$ we have
\[
\langle \mathcal{L}^{s}f,g\rangle=\mathcal{E}(f,g).
\]
The bilinear form $\mathcal{E}$ is well defined and for functions in the space  $\dot{\mathcal{H}}_{\mathcal{L}^{s}}(\mathbb{R}^{N})$, which is the closure of $C_0^\infty(\mathbb{R}^N)$, with the seminorm associated to the quadratic form $\overline{\mathcal{E}}$. We also define
\[
\label{def-HL}\mathcal{H}_{\mathcal{L}^{s}}(\mathbb{R}^{N})=\{f\in L^{2}%
(\mathbb{R}^{N})\,:\,\overline{\mathcal{E}}(f)<\infty\},
\]
which turns out to be a Dirichlet form whose completed domain we call $H\times H \subset \LL^2(\R^N)\times\LL^2(\R^N)$\,. Define the norm:
\begin{equation}\label{def.H.norm}
\|f\|_H=\left(\sum_{k=1}^{\infty}\lambda_k\hat{f}_k^2\right)^{{1}/{2}}<+\infty\qquad\mbox{with}\qquad \hat{f}_k=\int_{\R^N} f(x)\Phi_k(x)\ dx\,.
\end{equation}
Notice that the closure of the domain of the Dirichlet form of $ \mathcal{L}^{s} $ is given by
\begin{equation}\label{def.H.space}
H=H(\R^N):=\left\{f\in \LL^2(\R^N)\;|\; \sum_{k=1}^{\infty}\lambda_k\hat{f}_k^2<+\infty\right\}.
\end{equation}
The above function space  is a Hilbert space with the inner product given by the Dirichlet form
\[
\langle f, g \rangle_H = \sum_{k=1}^{\infty}\lambda_k\hat{f}_k\,\hat{g}_k= \mathcal{E}(f,g).
\]

\noindent We will also consider the dual space $H^*$ endowed with its dual norm,
\begin{equation}\label{def.H*.norm.1}
\|F\|_{H^*}=\sup_{\begin{subarray}{c}  g\in H \\ \|g\|_{H}\le 1\end{subarray}}\langle F, g \rangle_{H^*,H}
=\sup_{\begin{subarray}{c}  g\in H \\ \|g\|_{H}\le 1\end{subarray}}\sum_{k=1}^{\infty}\hat{F}_k\,\hat{g}_k\,,
\end{equation}
where $\langle \cdot\,,\, \cdot \rangle_{H^*,H}$ is the duality mapping. We have
\begin{equation}\label{def.H*.norm.2}
\|F\|_{H^*}=\left(\sum_{k=1}^{\infty}\lambda_k^{-1}\hat{F}_k^2\right)^{{1}/{2}}\,.
\end{equation}
So $ \mathcal{L}^{s} $ gives the canonical isomorphism between $H$ and $H^*$ \cite{BonforteSireVazquez}.\medskip

\begin{lemma}\label{lem:contraction-estimate}
Let the pair $(w, \vartheta)$ be solution to \eqref{Eq:main-2} in $[0,1]\times \R^{N}$ with right hand side $f \in C^{\beta}, w \in C^{\beta}\left([1/9,T] \times \R^{N}\right)$. Assume also that $\left|w^{m}(0,x)\right| \leq C_{1}$ and $\vartheta$ satisfies \eqref{Eq:main-2}. Then, for any increasing cut-off function $\eta \in C^{\infty}$ satisfying $\eta(t) = 0$ for $t < 1/2, ~\eta(t) = 1$ for $t>1$, we have that $\vartheta \eta$ is a solution to \eqref{Eq:main-2} in $(-\infty,T)$, with the right hand side $ \tilde{f} \in C^{\beta}$ satisfying
\[
\left|\tilde{f}(t) -\tilde{f}(s) \right| \leq C_{2} \left|t-s\right|^{\beta},
\]
as long as $\left|t-s\right| > \rho(\varepsilon)$, where $\rho(\varepsilon)$ is the modulus of continuity. Furthermore, the following contraction property holds true
\begin{equation}\label{Eq:contraction3}
\left|w(t,x) - w(s,x) \right| \leq C \left|t\right|^{2\beta},
\end{equation}
where the constant $C$ depends on the distance from $t$ and $s$ to the initial point $0$.
\end{lemma}

The proof of this lemma is given in Appendix~A.\medskip

Next we consider the abstract fractional differential inclusion for $\upgamma, s \in (0,1)$
\begin{equation}\label{Eq:AFDE.Problem}
\begin{cases}
\fcaputo w(t) + \mathcal{A}(t) \ni  f(t), &  ~ {\rm for ~all} \quad t>0,\\
w(0,x) = w_0(x), & ~{\rm in} \quad \R^N, 
\end{cases}
\end{equation}
with $\mathcal{A}(t,\cdot):=\mathcal{L}^{s}\varphi(w)$.\medskip

One can write the nonlocal fractional differential problem~\eqref{Eq:AFDE.Problem} as the functional differential equation
\begin{equation}\label{Eq:AFDE.Problem1}
\int_{-\infty}^{0}J(\theta) w'(t+\theta) \ d\theta = - \mathcal{A}(t) + f(t), \quad \text{and} ~~ w(\theta) = \rho(\theta), \quad ~\text{for} \quad \theta \leq 0 ,
\end{equation}
with initial value  where $t-\tau = -\theta \leq 0$, with the kernel $J(\theta) = J(-\theta)$, the even extension of $J$. One could notice that if we let $\rho(\theta) = w_{0}(x)$, $\theta \leq 0$ and $J(\theta) = \frac{\left|\theta\right|^{-\gamma}}{\Gamma(1-\gamma)} $, then equation~\eqref{Eq:AFDE.Problem1} reduces to equation~\eqref{Eq:AFDE.Problem}. In fact, for $\theta \to w(t+\theta)$ absolute continuous, we have
\[
\int_{-\infty}^{0}J(\theta) w'(t+\theta) \ d\theta = \int_{0}^{t} J(t-\tau) w'(\tau) \ d\tau.
\]
We then embed the solution $w(t) = U(t,0)$ in the ``history" state space
\[
U(t,\theta) = w(t+\theta) \in Z = C\left((-\infty,0], H^{*}\right).
\]
Hence equation~\eqref{Eq:AFDE.Problem} has the Markovian form as the evolution equation in $Z$:
\begin{equation}\label{Eq:AFDE.Problem2}
\frac{d}{dt}U(t) = \mathcal{B}(t)U(t),
\end{equation}
where the operator $\mathcal{B}(t)$ is defined by  
\begin{equation}\label{Eq:AFDE.Problem3}
\mathcal{B}(t) \rho = \rho'(\theta), \quad \theta \in (-\infty,0]
\end{equation}
in $Z$ with a domain
\[
\mathcal{D}\left(\mathcal{B}(t)\right) = \left\lbrace \rho \in Z: \rho' \in Z~~\text{and}~~\int_{-\infty}^{0}J(\theta) \rho'(\theta) \ d\theta \ni \mathcal{A} \rho(0) + f(t), \quad ~\rho(0) \in \mathcal{D}\left(\mathcal{A}\right)  \right\rbrace.
\]
So the dynamics of equation~\eqref{Eq:AFDE.Problem} is embedded in equation~ \eqref{Eq:AFDE.Problem2} as the nonlocal boundary value condition $\theta = 0^{+}$ for the first order differential operator \cite{Ito}. \medskip

\section{Existence and uniqueness. Proof of Theorem~\ref{theo:existence}}\label{sec:Existence_Uniqueness} 
We analyze the well-posedness and property of the solution to equation~\eqref{Eq:main_problem} or equation~\eqref{Eq:AFDE.Problem} based on Crandall-Liggett theory for nonlinear monotone graph or implicitly based on the semi-group generated by equation~ \eqref{Eq:AFDE.Problem2}, i.e., we should show that the solution map $(w_{0}(x),f) \in H^{*} \times C(0,T;H^{*}) \to w(t) \in C(0,T;H^{*})$ exists and is continuous. \medskip

So if $\mathcal{A}$ is maximal monotone operator in $H^{*}$, it will be shown that $\mathcal{B}(t)$ is maximal monotone in $Z$ as well, hence $m$-acretive. Thus $w$ satisfies \eqref{Eq:main_problem}. After that stage we proceed with the approximating procedure to show that the solution $(w,\vartheta)$ to the weak formulation \eqref{Eq:weak} holds true as well.\medskip

The proof can be adapted to our fractional operator in time once the proper functional analysis is in place.  Let us begin by setting up some notations borrowed from \cite{BonforteSireVazquez}. \medskip

Let $j$  be a convex, lower semi-continuous function $j:\R\to \R$ and such that $j(r)/|r|\to \infty$ as $|r|\to\infty$. We let $\varphi=\partial j$ be the sub-differential of $j$.  For $w\in H^*(\R^N)$ we define
\[
\Psi(w)=\int_{\R^{N}} j(w)\,\ dx
\]
whenever $w\in\LL^1(\R^N)$ and $j(w)\in\LL^1(\R^N)$, and define $\Psi(w)=+\infty$ otherwise. The example we have in mind is  $\varphi(w)=|w|^{m-1}w$ and $j(w)=|w|^{m+1}/(m+1)$\,, so that $\Psi(w)=\|w\|_{\LL^{m+1}(\R^N)}/(m+1)$\,.\medskip

Let $H^{\dagger}$ the dual space of $H^{*}$, $\langle \cdot,\cdot \rangle_{H^{*} \times H^{\dagger}}$ denotes the dual product and $F$ be the duality mapping
\[
F(x) = \left\lbrace x^{\dagger} \in H^{\dagger}: \left|\langle x, x^{\dagger}\rangle \right| = \left|x\right|_{H^{*}}\left|x^{\dagger}\right|_{H^{\dagger}} \right\rbrace.
\]

From \cite[Proposition 3.1]{BonforteSireVazquez}, $\Psi$ is convex and lower semi-continuous function in $H^*$, so that its sub-differential $\partial\Psi$ is a maximal monotone operator in $H^*$. In this way, $\mathcal{A} \subset H^{*} \times H^{\dagger} $ be maximal monotone graph. Moreover, the sub-differential $\partial\Psi$ can be characterized as, for $[x_{i},y_{i}] \in \mathcal{A}$ there exists $x^{\dagger} \in F(x_{1},x_{2})$ such that $\mathcal{R}e \langle y_{1}- y_{2}, x^{\dagger} \rangle \leq 0$, where $F$ is the duality mapping defined as above.\medskip

\begin{proof}
Following the arguments in \cite{BonforteSireVazquez, Brezis, Ito} the proof of Theorem~\ref{theo:existence} will be split in two steps.\medskip

\noindent $\bullet~$\textsc{Step 1. } We first need to prove that if $\mathcal{A}$ is a maximal monotone graph and Range $\left(\lambda I - \mathcal{A} \right)$ for an interval $I \in \left(-\infty, 0 \right)$ for all small $\lambda > 0$, then the operator $\mathcal{B}$ is a maximal monotone graph and Range $\left(\lambda I - \mathcal{B}\right) = Z$    for all small $\lambda > 0$, so that $\mathcal{B}$ generates the nonlinear semigroup of contraction on $Z$. We essentially adapt the proof of \cite[Theorem 3.1]{Ito}. \medskip

Let  $\mathcal{B}$ in $Z = C \left((-\infty, 0];H^{*}\right)$ so that 
$\mathcal{B} \rho = \rho'$ as in \eqref{Eq:AFDE.Problem3}. For $\rho_{1}, \rho_{2} \in \mathcal{D}(\mathcal{B})$, suppose
\[
\left|\rho_{1}(0) - \rho_{2}(0) \right| > \left|\rho_{1}(\theta) - \rho_{2}(\theta) \right| \quad~\text{for all}~ \theta < 0.
\]

In that setting for all $x^{\dagger} \in F \left(\rho_{1}(0) - \rho_{2}(0) \right)$ we have for $\varepsilon$ small enough,
\begin{multline}
\bigg\langle \int_{-\infty}^{0}J_{\varepsilon}(\theta) \left(\rho'_{1} - \rho'_{2} \right) d\theta ~,~ x^{\dagger} \bigg\rangle \\
= \int_{-\infty}^{0} \frac{J_{\varepsilon}(\theta) - J_{\theta -\varepsilon}(\theta)}{\varepsilon}\bigg\langle  \rho_{1}(\theta) - \rho_{2}(\theta)  - \left( \rho_{1}(0) - \rho_{2}(0)  \right)~,~x^{\dagger} \bigg\rangle d\theta \leq 0.
\end{multline}

Now notice that
\begin{multline}
\bigg\langle  \rho_{1}(\theta) - \rho_{2}(\theta)  - \left( \rho_{1}(0) - \rho_{2}(0)  \right)~,~x^{\dagger} \bigg\rangle \\
 \leq \left( \left| \rho_{1}(\theta) - \rho_{2}(\theta) \right| - \left| \rho_{1}(0) - \rho_{2}(0)\right|\right)\left| \rho_{1}(0) - \rho_{2}(0)\right| < 0, \qquad~~\theta <0.
\end{multline}
Hence 
\begin{equation}\label{Eq:AFDE.Problem4}
\bigg\langle \int_{-\infty}^{0}J_{\varepsilon}(\theta) \left(\rho'_{1} - \rho'_{2} \right) d\theta ~,~ x^{\dagger} \bigg\rangle < 0. 
\end{equation}

On the other hand as $y_{1} \in \mathcal{A}\rho_{1}(0)$ and $y_{2} \in \mathcal{A}\rho_{2}(0)$, then there exists a vector $x^{\dagger}$ in the space $F \left( \rho_{1}(0) - \rho_{2}(0) \right)$ such that $ \langle y_{1} - y_{2}~,~x^{\dagger}\rangle \geq 0$, which is in contradiction with the inequality ~\eqref{Eq:AFDE.Problem4}. Thus there exists $\theta_{0}$ such that 
\[
\left| \rho_{1}(\theta_{0}) - \rho_{2}(\theta_{0})\right| = \left| \rho_{1} - \rho_{2}\right|Z,
\]
so that 
\[
\langle \rho'(\theta_{0})~,~x^{\dagger}\rangle = 0 \quad \text{for all} ~~ x^{\dagger} \in F\left(\rho_{1}(\theta_{0}) - \rho_{2}(\theta_{0})\right).
\]
Thus 
\[
\begin{aligned}
\left|\lambda\left(\rho_{1} - \rho_{2} \right) - \left(\rho'_{1} - \rho'_{2} \right)\right|_{Z} &\geq \big \langle \lambda\left(\rho_{1}(\theta_{0}) - \rho_{2}(\theta_{0}) \right) - \left(\rho'_{1}(\theta_{0}) - \rho'_{2}(\theta_{0}) \right)~,~ x^{\dagger} \big \rangle\\
&= \lambda \left| \rho_{1}(\theta_{0}) - \rho_{2}(\theta_{0}) \right|= \lambda \left| \rho_{1} - \rho_{2} \right|_{Z}.
\end{aligned}
\]
Next as noticed in \cite{Ito}, for Range condition
\[
\lambda \rho - \rho' = f, \quad \int_{-\infty}^{0}J(\theta) \rho'(\theta)\ d\theta \in \mathcal{A} \rho(0),
\]
we have
\[
\rho =e^{\lambda \theta} \rho(0) + q
\quad \text{and} \quad
\left( R(\lambda)\rho\right)(0) - \int_{-\infty}^{0}J(\theta) q'(\theta)\ d\theta \in \mathcal{A} \rho(0),
\]
where the functions $R(\lambda)$ and $q$ are respectively defined as
\[
R(\lambda) = \lambda\int_{-\infty}^{0} e^{\lambda \theta}J(\theta) \ d\theta, \quad \text{and} \quad
q(\theta) = \int_{\theta}^{0} e^{\lambda(\theta-\xi)}f(\xi) \ d\xi,  \qquad \vert q \vert \leq \frac{1}{\lambda} \vert f \vert_{Z}.
\]
Since $\mathcal{A}$ is a maximal monotone operator, then
\[
\rho(0) = \left(R(\lambda)I - \mathcal{A} \right)^{-1} \int_{-\infty}^{0}J(\theta)q'(\theta) \ d\theta
\]
exists and Range $\left( \lambda I - \mathcal{A} \right) = Z$. Hence the result follows from Crandall-Liggett Theorem \cite{Crandall-Liggett} and from \cite[Theorem 4.1]{Ito} which holds in general. \medskip

\noindent $\bullet~$\textsc{Step 2. } 
In this second step, we consider the equivalent problem \eqref{Eq:main-vartheta}, which is more convenient to handle. After the construction of solution by means of maximal monotone operators and Crandall-Liggett Theorem \cite{Crandall-Liggett} in the first step, we prove the existence of solution $(w,\vartheta)$ to the weak formulation \eqref{Eq:weak} via the approximating solutions. Basically the idea is to approximate the problem by a pair of sequences of bounded integrable functions $(w_{\varepsilon},\vartheta_{\varepsilon})$  to \eqref{Eq:main-2} and then pass to the limit in the approximate problem. The key tools needed to pass to the limit are the $L^{1}$-contraction property, Lemma~\ref{lem:contraction-estimate} and the smoothing effect. \medskip

We define the operators in \eqref{Eq:weak} and \eqref{Eq:main-2} as $\mathcal{H}$ and $\mathcal{H}_{\e}$ respectively. We also denote by $\K_{\e}(t,x) = \K(\e j, x)$ if $\e (j-1) < t< \e j$ and $\mathcal{E}_{\e}$ the discrete bilinear form associated with the kernel $\K_{\e}$. Using approximating solutions as in \cite{Allen1}, the proof is concluded once we have the following result
\[
\displaystyle \mathcal{\overline{H}}(w,\vartheta,\psi):= \mathcal{H}(w,\vartheta,\psi) + \mathcal{H}_{\e}(w,\vartheta,\psi) \to 0,
\]
for $\psi \in C_0^1((a,T)\times \R^N)$ with $\psi(t,x) = \psi(\varepsilon j,x)$. We refer to the Appendix B for the details on these computations.\medskip
\end{proof}


\section{Regularity of solutions. Proof of Theorem~\ref{theo:regularity}}\label{sec:regularity}
In this section we prove Theorem~\ref{theo:regularity}. The main ingredient of the proof are the De Giorgi Lemmas, commonly known as the ``$L^{2}, L^{\infty}$" estimates. These techniques have been previously used e.g. in \cite{Allen1, Pablo3, CaffarelliChanVasseur}.\medskip

The nonlinearity part is confined into the fractional time derivative operators, we will need to establish some energy estimates for the weak solutions of \eqref{Eq:main-vartheta}.\medskip

\subsection{The first De Giorgi's Lemma}\label{firstdegiorgi}

We first introduce the following Lipschitz function $\psi$:
\[
\psi(t,x):=\left(\vert t \vert^{\gamma/2} - 1 \right)_{+} + \left( \vert x \vert^{s/2} - 1 \right)_{+},
\]
so that for $L \geq 0$, we define
\begin{equation}\label{eq:psiL}
\psi_{L}(t,x) = L + \psi(t,x).
\end{equation}
Let us introduce (see e.g. \cite{Pablo3}) $\ell=\inf_{\{w-\psi_{L} \ge 0\}}w\ge0$ and $M=\sup_{\{w-\psi_{L} \ge 0\}}w<\infty$. \medskip

In the above settings for $a+\varepsilon (j-1) < t \leq a+\varepsilon j$ we extend $w(t) = w(a+\varepsilon j)$, $\vartheta(t) = \vartheta(a+\varepsilon j)$, and $\psi(t)=\psi(a + \varepsilon j)$ . \medskip

The first step of the regularity argument is to obtain an energy estimate associated to $\partial_{\varepsilon}^{\upgamma}\vartheta_{\varepsilon}(w)$ defined above.\medskip

\begin{lemma}\label{lem:energy1}
Let $\psi_{L}$ be defined in \eqref{eq:psiL} with $L \geq 0$, satisfying
\eqref{Eq:kernel_laplacian} for every $x\in\R^N$, and $w$ be a weak solution to~\eqref{Eq:main-vartheta} in some finite time interval $I$ including $(t_1,t_2)$ or in the discrete form $(\varepsilon j_1,\varepsilon j_2)$. Then,
\begin{multline}\label{energy3}
\int_{\R^{N}} \mathcal{B}_{\psi_{L}}(w)(a + \varepsilon j_{2}, x) + \sum_{0 < j_{1} < j_{2} < k}\overline{\mathcal{E}}((w-\psi_{L})_{+})(a + \varepsilon j, x) \\
\leq \int_{\R^{N}} \mathcal{B}_{\psi_{L}}(w)(a + \varepsilon j_{1}, x) + \int_{\R^{N}} \sum_{0 < j_{1} < j_{2} < k}\varepsilon (w-\psi_{L})_{+}(a + \varepsilon j, x) f(a + \varepsilon j, x)\\
+ ~C\sum_{0 < j_{1} < j_{2} < k}\left( \int_{\R^{N}}
(w-\psi_{L})_{+}(a + \varepsilon j, x)+\chi_{\{w(x,t)-\psi_{L}(x)>0\}}\right).
\end{multline}
\end{lemma}
\begin{proof}
We first consider the discrete form of~\eqref{Eq:main-vartheta} given as
\begin{equation}\label{Eq:discretemainproblem2}
\partial_{\varepsilon}^{\upgamma}\vartheta_{\varepsilon}(w)(a + \varepsilon j, x)+\mathcal{K}w(a + \varepsilon j, x) = f(a + \varepsilon j, x).
\end{equation}
We define $v:= M-w$. According to~ \eqref{Eq:discretemainproblem2} we have
\[
\partial_{\varepsilon}^{\upgamma}\vartheta_{\varepsilon}(M-v)(a + \varepsilon j, x)+\mathcal{K}\left(M-v\right)(a + \varepsilon j, x) = f(a + \varepsilon j, x).
\]
Choose $\psi_{L} \geq 0$ and a smooth cutoff function $\zeta$ vanishing near the parabolic boundary of $Q$. We multiply the above equation by the test function $\varphi = \zeta(v-\psi_{L})_{+}$, and integrate over $Q$, so that
\begin{multline*}
\sum_{0<j<k} \int_{\mathbb{R}^{N}} \zeta(v-\psi_{L})_{+} \partial_{\varepsilon}^{\upgamma}\vartheta_{\varepsilon}(M-v)(a + \varepsilon j, x)~ dx + \sum_{0<j_{1}<j_{2}<k}\mathcal{E}(w,(w-\psi_{L})_{+})(a + \varepsilon j, x) \\
=  \int_{\mathbb{R}^{N}} \sum_{0<j_{1}<j_{2}<k} \zeta (v-\psi_{L})_{+}(a + \varepsilon j, x)f(a + \varepsilon j, x).
\end{multline*}
By using integration by parts, it yields
\begin{multline*}
\left. \int_{\mathbb{R}^{N}} \int_{|\tau|\leq 1}(v-\psi_{L})_{+} \partial_{\varepsilon}^{\upgamma}\vartheta_{\varepsilon}(M-v)(a + \varepsilon j, x)~d\nu \right| _{j=j_{1}
}^{j=j_{2}}~ dx \\
+ \sum_{0<j_{1}<j_{2}<k}\mathcal{E}(w,(w-\psi_{L})_{+})(a + \varepsilon j, x) =  \int_{\mathbb{R}^{N}} \sum_{0<j_{1}<j_{2}<k} \zeta (v-\psi_{L})_{+}(a + \varepsilon j, x)f(a + \varepsilon j, x).
\end{multline*}
Let $(v-\psi_{L})_{+} = \tau$. Then 
\begin{multline*}
\left. \int_{\mathbb{R}^{N}} \int_{|\tau|\leq 1} \partial_{\varepsilon}^{\upgamma}\vartheta_{\varepsilon}(M-\psi_{L}-\tau)(a + \varepsilon j, x)\tau~d\tau \right| _{j=j_{1}
}^{j=j_{2}}~ dx \\
+ \sum_{0<j_{1}<j_{2}<k}\mathcal{E}(w,(w-\psi_{L})_{+})(a + \varepsilon j, x) =  \int_{\mathbb{R}^{N}} \sum_{0<j_{1}<j_{2}<k} \varepsilon (v-\psi_{L})_{+}(a + \varepsilon j, x)f(a + \varepsilon j, x).
\end{multline*}
Setting  
\begin{equation}\label{Eq:functional}
\mathcal{B}_{\psi_{L}}(v)=\int_{0}^{(v-\psi_{L})_{+}}\partial_{\varepsilon}^{\upgamma}\vartheta_{\varepsilon}(M-\psi_{L}-\tau)\tau d\tau,
\end{equation}
we get the following energy inequality
\begin{multline}\label{weak-w}
\left. \int_{\mathbb{R}^{N}}\mathcal{B}_{\psi_{L}}(w(a + \varepsilon j, x)) \right| _{j=j_{1}
}^{j=j_{2}}  +\sum_{0<j_{1}<j_{2}<k}\mathcal{E}(w(a + \varepsilon j, x),(w-\psi_{L})_{+})(a + \varepsilon j, x) \\
= \int_{\mathbb{R}^{N}} \sum_{0<j_{1}<j_{2}<k} \varepsilon (w-\psi_{L})_{+}(a + \varepsilon j, x)f(a + \varepsilon j, x).
\end{multline}
However, since we do not know yet whether $\vartheta(w)$ has the required time regularity, a regularization procedure in the weak
formulation using some Steklov averages and discretization procedure (see e.g. \cite{Pablo2, Pablo3, Crandall-Liggett, Benilan1, Aronson-Serrin}), allows to bypass this difficulty, since in that form $\vartheta_{\varepsilon}(w)$ will be more regular. In fact, it suffices to show that
$\partial_{t}^{\gamma}\mathcal{B}_{\psi_{L}}(w)\in L^{2}_{\mathrm{loc}}\left((-\infty,0) \times L^{2}(\mathbb{R}^{N})\right)$.

For any $v\in L^{1}(I \times \mathbb{R}^N)$  and $0<h<t_{2}-t_{1}$, we
define the Steklov average
\[
v^{h}(t,x)=\frac1h\int_{t}^{t+h}v(\tau,x)\,d\tau.
\]
Fix $t \in I$, and $h>0$ such that $t+h \in I$. So we have that
\[
\partial_{t}v^{h}(x,t)=\delta^{h}v(x,t):=\frac{v(x,t+h)- v(x,t)}{h}\,.
\]
Since $\partial_{\varepsilon}\vartheta(w)^{h}\in L^{1}(I \times \mathbb{R}^{N})$, we can obtain the weak formulation~\eqref{weak-nonlocal} in the form
\begin{multline*}
\int_{\mathbb{R}^{N}}\sum_{0<j\leq k} \varphi(a + \varepsilon j, x)~\partial_{\varepsilon}^{\gamma}\vartheta_{\varepsilon}(w)^{h}(a + \varepsilon j, x) - \sum_{0<j\leq k}\varepsilon \mathcal{E}(w^{h},\varphi)(a + \varepsilon j, x) \\=  \int_{\mathbb{R}^{N}} \sum_{0<j\leq k}\varphi(a + \varepsilon j, x)~f(a + \varepsilon j, x).
\end{multline*}
We take $\varphi = \left(\zeta \partial_{\varepsilon} w^{h} \right)^{-h}$ as test function, where $\zeta \in
C_{0}^{\infty}(I)$, $0\leq\zeta\leq 1$, $\zeta(\varepsilon j)=1$ for $\varepsilon j \in[\varepsilon j_{1}
,\varepsilon j_{2}]$, is a cut-off function. Then we use integration by parts \cite{Pablo3} over some time interval  $(\varepsilon j_{1}, \varepsilon j_{2})$, to obtain
\begin{multline*}
\int_{\mathbb{R}^{N}}\sum_{0<j\leq k} \zeta ~\partial_{\varepsilon}^{\gamma}\vartheta_{\varepsilon}(w)^{h}(a + \varepsilon j, x)~\partial_{\varepsilon} w^{h}(a + \varepsilon j, x) - \sum_{0<j\leq k} \varepsilon  \zeta \mathcal{E}(w^{h},\partial_{\varepsilon} w^{h})(a + \varepsilon j, x)\\
 =  \int_{\mathbb{R}^{N}} \sum_{0<j\leq k}\zeta \partial_{\varepsilon} (w)^{h}~f(a + \varepsilon j, x).
\end{multline*}
{}From Lemma~\ref{lem:switch}, we can write
\[
\begin{aligned}
&\int_{\mathbb{R}^{N}}\sum_{0<j\leq k} \zeta ~\partial_{\varepsilon}^{\gamma}~\partial_{\varepsilon}^{-1}
~\left( \partial_{\varepsilon}\vartheta_{\varepsilon}(w)^{h}~\partial_{\varepsilon} w^{h} \right)(a + \varepsilon j, x)\\
& = \frac{1}{2} \sum_{0<j\leq k} \varepsilon  \zeta \mathcal{E}(w^{h}, w^{h})(a + \varepsilon j, x)
 +  \int_{\mathbb{R}^{N}} \sum_{0<j\leq k}\zeta \partial_{\varepsilon} (w)^{h}~f(a + \varepsilon j, x)\\
& = -\frac{1}{2} \sum_{0<j\leq k} \partial_{\varepsilon} \zeta(t) \mathcal{E}(w^{h}, w^{h})(a + \varepsilon j, x)
 +  \int_{\mathbb{R}^{N}} \sum_{0<j\leq k}\zeta \partial_{\varepsilon} (w)^{h}~f(a + \varepsilon j, x).
\end{aligned}
 \]
Now we observe that the same inequality used in \cite{Pablo1, Pablo2, Pablo3, Brandle-dePablo} allows to prove that $\delta_{\varepsilon}^{h} \vartheta_{\varepsilon}(w)\,\delta_{\varepsilon}^{h} w\geq(\delta_{\varepsilon}^{h}(\ell(w)))^{2}$, where $(\ell^\prime)^2=\vartheta_{\varepsilon}^{\prime}$. By using this inequality
\begin{multline*}
\int_{\mathbb{R}^{N}}\sum_{0<j\leq k} \zeta ~\partial_{\varepsilon}^{\gamma-1}
~\left( \partial_{\varepsilon}\vartheta_{\varepsilon}(w)^{h}~\partial_{\varepsilon} w^{h} \right)(a + \varepsilon j, x)\\
 \leq \sum_{0<j\leq k} \vert \partial_{\varepsilon} \zeta^{\prime} \vert  \mathcal{E}(w^{h}, w^{h})(a + \varepsilon j, x)
 +  \int_{\mathbb{R}^{N}} \sum_{0<j_{1}<j_{2}\leq k}\zeta \partial_{\varepsilon} (w)^{h}~f(a + \varepsilon j, x).
\end{multline*}
By using Lemma~\ref{lem:switch} we have that $\partial_{\varepsilon}^{\gamma-1} h(a + \varepsilon j, x) \geq -c_{\nu,\gamma-1} $. On the other hand since $\partial_{\varepsilon} (w)^{h} \in L^{1} \left(I \times \R^N \right)$, then 
\[
\int_{\mathbb{R}^{N}} \sum_{0<j_{1}<j_{2}\leq k}\zeta \partial_{\varepsilon} (w)^{h}~f(a + \varepsilon j, x) \leq C~\int_{\mathbb{R}^{N}} \sum_{0<j_{1}<j_{2}\leq k} \partial_{\varepsilon} (w)^{h}~f(a + \varepsilon j, x) \leq C_{1}.
\]
Thus,
\begin{multline*}
\int_{\mathbb{R}^{N}}\sum_{0<j_{1}<j_{2}\leq k}  ~\partial_{\varepsilon}^{\gamma-1}
~\left( \partial_{\varepsilon}\vartheta_{\varepsilon}(w)^{h}~\partial_{\varepsilon} w^{h} \right)(a + \varepsilon j, x)\\
 \leq C_{\gamma}\sum_{0<j_{1}<j_{2}\leq k} \vert \partial_{\varepsilon} \zeta^{\prime} \vert  \mathcal{E}(w^{h}, w^{h})(a + \varepsilon j, x)
 + C_{\gamma} \int_{\mathbb{R}^{N}} \sum_{0<j_{1}<j_{2}\leq k}(w)^{h}~f(a + \varepsilon j, x)
 \leq C(s,\gamma,\Lambda,N ).
\end{multline*}
Furthermore from \cite{Pablo3} it can be shown that $|\delta_{\varepsilon}^{h}\mathcal{B}_{\psi_{L}}(w)| \leq |\sqrt{\vartheta^{\prime}(w)}w\delta_{\varepsilon}
^{h}\ell(w)|$, so $\delta_{\varepsilon}^{h}\mathcal{B}_{\psi_{L}}(w)\in L^{2}(I:L^{2}(\mathbb{R}^{N}))$ provided $\sqrt{\vartheta^{\prime}(w)}w\in L^\infty(\mathbb{R}^N\times I)$, and the proof ends by passing to the limit as $h\to0$.
\end{proof}

\begin{lemma}\label{lem:firstdeGiorgi}
Let $\Lambda$ be a given elliptic constant in the conditions~\eqref{Eq:kernel_laplacian} and let $w$ be a solution to \eqref{Eq:main-vartheta}. Then, there exists a constant $\sigma_{0} \in (0,1)$, depending only on $N,s,\Lambda, \gamma$ ---but independent of $\varepsilon$ and $a$--- such that for any solution $w:[a,0]\times \R^N \hookrightarrow \R$ of \eqref{Eq:main-vartheta} with $\Norm{f}_{L^{\infty}(Q)} \leq 1$ and $a \leq -1$, satisfying
\[
\int_{a}^{T} \int_{\R^N} \left[w(t,x) - \psi(t,x) \right]^{2}_{+} dx \ dt \leq \sigma_{0}, \quad w(a,x) \leq \psi(a,x) \quad \text{for all} ~~ x \in \R^N,
\] 
then
\[
w(t,x) \leq \frac{1}{2} + \psi(t,x),
\]
for $(t,x) \in [a,T] \times \R^N$. In these settings, $w \leq \frac{1}{2}$ on $[-1,0] \times B_{1}(0)$.
\end{lemma}
\begin{proof}
We split the proof into several steps. \medskip

$\bullet$ {\sc First step: Energy estimates}. Let $w:[a,0]\times \R^N \hookrightarrow \R$ solution to \eqref{Eq:discretemainproblem2}. For $\varepsilon = T/k$, and for $0 \leq L \leq 1$, we consider the truncated function $\left[ w - \psi_{L} \right]_{+}$. Then we take the test function $\varphi$ to be $\varepsilon \left[ w - \psi_{L} \right]_{+}$ and integrate over $\R^N$ to obtain
\begin{multline}\label{Eq:formulation:energy2}
\int_{\R^N} \sum_{0<j\leq k}\varepsilon \left( w - \psi_{L} \right)_{+}(a + \varepsilon j, x) \partial_{\varepsilon}^{\gamma}\vartheta_{\varepsilon}(w)(a + \varepsilon j, x) 
 + \sum_{0<j\leq k} \varepsilon \mathcal{E}(w(a + \varepsilon j, x),(w-\psi_{L})_{+})(a + \varepsilon j, x) \\
= \int_{\R^N} \sum_{0<j\leq k}\varepsilon \left( w - \psi_{L} \right)_{+}(a + \varepsilon j, x) f(a + \varepsilon j, x).
\end{multline}
{}From Lemma \ref{lem:energy1}, equation~\eqref{Eq:formulation:energy2} obeys the energy inequality 
\begin{multline}\label{Eq:formulation:energy3}
\int_{\R^{N}} \mathcal{B}_{\psi_{L}}(w)(a + \varepsilon j_{2}, x) + \sum_{0 < j_{1} < j_{2} < k}\overline{\mathcal{E}}((w-\psi_{L})_{+})(a + \varepsilon j, x) \\
\leq \int_{\R^{N}} \mathcal{B}_{\psi_{L}}(w)(a + \varepsilon j_{1}, x) + C_{\gamma}\int_{\R^{N}} \sum_{0 < j_{1} < j_{2} < k}\varepsilon (w-\psi_{L})_{+}(a + \varepsilon j, x) f(a + \varepsilon j, x)\\
+ ~C\sum_{0 < j_{1} < j_{2} < k}\left( \int_{\R^{N}} (w-\psi_{L})_{+}(a + \varepsilon j, x)+\chi_{\{w(x,t)-\psi_{L}(x)>0\}}\right).
\end{multline}
As a consequence the full energy estimate is obtained by using the properties of $\vartheta$, $\partial_{\varepsilon}\vartheta$ and $\partial_{\varepsilon}^{\gamma}\vartheta$ given in \eqref{Eq:properties_beta}. In this way, the following bounds are obtained
\begin{equation*}
\mathcal{B}_{\psi_{L}}(w)=\int_{0}^{(w-\psi_{L})_{+}}\vartheta^{\prime}(\tau+\psi_{L})\tau d\tau \geq \frac{c_{1}}{(1-\upgamma)\Gamma(1-\upgamma)}\int_{0}^{(w-\psi_{L})_{+}} \tau^{2-\upgamma} d \tau 
\geq c_{1}(\upgamma)(w-\psi_{L})_{+}^{3-\upgamma},
\end{equation*}
and
\begin{multline*}
\mathcal{B}_{\psi_{L}}(w)=\int_{0}^{(v-\psi_{L})_{+}}\partial_{\varepsilon}^{\upgamma}\vartheta_{\varepsilon}(M-\psi_{L}-\tau) d\tau \leq \frac{c_{1}}{(1-\upgamma)\Gamma(1-\upgamma)}\int_{0}^{(v-\psi_{L})_{+}} \tau^{1-\upgamma} d \tau \\
 \leq c_{2}(c_{1},\upgamma)(w-\psi_{L})_{+}^{2-\upgamma}.
\end{multline*}
Hence, we get the bounds of the functional $\mathcal{B}_{\psi_{L}}(w)$ given as 
\begin{equation*}
\label{estimate-B}\Lambda_{1}(w-\psi_{L})_{+}^{2}\le \mathcal{B}_{\psi_{L}}(w)\le \Lambda_{2} (w-\psi_{L})_+,
\end{equation*}
where
\begin{equation}\label{eq:lambdas}
\Lambda_1=\frac12\inf_{\ell\le \tau \le M}\vartheta'(\tau),\qquad \Lambda_2=\vartheta(M)-\vartheta(\ell).
\end{equation}
In this scenario the energy estimate~\eqref{Eq:formulation:energy3} yields
\begin{multline}\label{Eq:formulation:energy4}
\Lambda_{1}\int_{\R^{N}} \varepsilon(w-\psi_{L})_{+}^{3-\upgamma}(a + \varepsilon j_{2}, x) + C_{s,\gamma,\Lambda,N} \int_{\R^{N}} \sum_{0 < j_{1} < j_{2} < k} \left\vert \left(-\Delta \right)^{\frac{s}{2}} \left((w-\psi_{L})_{+}\right)(a + \varepsilon j, x) \right\vert^{2} \\
+ ~C\sum_{0 < j_{1} < j_{2} < k}\left( \int_{\R^{N}}
(w-\psi_{L})^{2}_{+}(a + \varepsilon j, x)+(w-\psi_{L})_{+}(a + \varepsilon j, x)+\chi_{\{w(x,t)-\psi_{L}(x)>0\}}\right) \\
\leq \Lambda_{2}\int_{\R^{N}} \varepsilon(w-\psi_{L})_{+}^{2-\upgamma}(a + \varepsilon j_{1}, x) + C_{\gamma} \int_{\R^{N}} \sum_{0 < j_{1} < j_{2} < k}\varepsilon (w-\psi_{L})_{+}(a + \varepsilon j, x) f(a + \varepsilon j, x).
\end{multline}
We denote $v=(w-\psi_{L})_{+}$. We also define 
\[
 \int_{\R^{N}} \sum_{0 < j_{1} < j_{2} < k}\varepsilon (w-\psi_{L})_{+}(a + \varepsilon j, x) f(a + \varepsilon j, x) =  \int_{\R^{N}} \int_{a}^{T} f~v,
\]
as well as 
\[
\int_{\R^{N}} \sum_{0 < j_{1} < j_{2} < k} \left\vert \left(-\Delta \right)^{\frac{s}{2}} \left((w-\psi_{L})_{+}\right)(a + \varepsilon j, x) \right\vert^{2} = \int_{t_{1}}^{t_{2}} \frac{1}{\Lambda} \norm{v}^{2}_{H^{\frac{s}{2}}(\R^N)}.
\]
Since $|f| \leq 1$, the term $\left(|f| + 1 \right)v$ is controlled by $2v$. Hence~\eqref{Eq:formulation:energy4} becomes 
\begin{multline}\label{Eq:formulation:energy5}
\Lambda_{1}\int_{\R^{N}}v^{3-\upgamma}(t_{2}, x) dx + C_{s,\Lambda,N} \int_{t_{1}}^{t_{2}} \frac{1}{\Lambda} \Norm{v}^{2}_{H^{\frac{s}{2}}(\R^N)} \\
\leq \Lambda_{2}\int_{\R^{N}} v^{2-\upgamma}(t_{1}, x)dx 
+ ~C_{s,\gamma,\Lambda, N}\int_{t_{1}}^{t_{2}} \int_{\R^{N}}  \left[ v^{2}+ v +\chi_{\{v>0\}} \right].
\end{multline}

$\bullet$ {\sc Second step: Nonlinear recurrence}\medskip

{}From the energy inequality, we establish a nonlinear recurrence relation to the following sequence of truncated energy
\[
U_{k}:=\sup_{t\in [T_{k},0]} \int_{\R^{N}}\Norm{(w-\psi_{L_{k}})_{+}(t,x)}^{3-\upgamma} \ dx  + \int_{T_{k}}^{0}  \Norm{(w-\psi_{L_{k}})_{+}(t,\cdot)}^{2}_{H^{\frac{s}{2}}(\R^N})\ dt,
\]
where we denote $T_{k} = -(1+\frac{1}{2^{k}})$, $L_{k} = \frac{1}{2}\left(1-\frac{1}{2^{k}}\right)$, and the cylinder domain $Q_{k} = \left[T_{k},0 \right] \times \R^N$.  Furthermore, let us consider two times variables $t_{1}, t_{2}$ that satisfy $T_{k-1} \leq t_{1} \leq T_{k} \leq t_{2} \leq 0$. By taking the time integral over $I=[t_{1},t_{2}]$, we obtain~\eqref{Eq:formulation:energy5}. \medskip

Next by taking average over $t_{1} \in [T_{k-1},T_{k}]$ and then taking the supremum over $t_{2}\in [T_{k},0]$ in ~\eqref{Eq:formulation:energy5}, we deduce that 
\begin{equation}\label{Eq:formulation:energy6}
U_{k} \leq 2^{k}~C_{s,\upgamma,\Lambda, N} \frac{ \left(1 + \Lambda_{2} \right)}{\Lambda_{1}}\left[\int_{t_{1}}^{t_{2}} \int_{\R^{N}} 
(w-\psi_{L_{k}})_{+}^{2}+ (w-\psi_{L_{k}})_{+} +\chi_{\{w-\psi_{L_{k}}>0\}} dx \ dt \right].
\end{equation}
Now we use the Sobolev embedding for fractional spaces \cite{Hitch2012, Allen1, CaffarelliChanVasseur, Luis_sylvestre}. From \cite{Hitch2012} and \cite{Allen1} we can write the following relation and the Sobolev embedding equivalence 
\[
\Norm{(w-\psi_{L_{k}})_{+}}^{2}_{\left(\R\right)} \geq C(\upgamma) \Norm{(w-\psi_{L_{k}})_{+}}^{2}_{H^{\frac{\upgamma}{2}}(\R)}~\text{and}~ H^{\frac{s}{2}}(\R^N) \subset L^{\frac{2N}{N-s}}
\left(\R^N\right),
\]
so that with an interpolation we obtain  
\begin{equation}\label{Eq:formulation:energy7}
\Norm{(w-\psi_{L_{k}})_{+}}^{2}_{L^{p^{*}}\left(Q_{k}\right)} \leq 2^{k}C_{1}\left[\int_{t_{1}}^{t_{2}} \int_{\R^{N}} 
(w-\psi_{L_{k}})_{+}^{2}+ (w-\psi_{L_{k}})_{+} +\chi_{\{w-\psi_{L_{k}}>0\}} dx \ dt \right],
\end{equation}
with
\[
p^{*} = \left( \frac{s + \upgamma N }{(1-\upgamma)s + \upgamma N } \right)~~\text{and}~~ C_{1}:=~C_{s,\upgamma,\Lambda, N} \frac{ \left(1 + \Lambda_{2} \right)}{\Lambda_{1}}.
\]
Since $L_{k} = L_{k-1} + 2^{-k-1}$, we have that $w_{k-1} > 2^{-k-1}$. Using the Chebyshev type inequality
\[
\int_{\R^N} w_{k}^{q} \leq 2^{(k+1)(p^{*}-q)}\int_{\R^N}w_{k-1}^{p^{*}},
\]
for every $p^{*}>q$, the following inequalities are obtained
\begin{equation*}
\begin{cases}
\displaystyle     \int_{Q_{k-1}} (w-\psi_{L_k})_+^2 & \displaystyle \leq \int_{Q_{k-1}}(w-\psi_{L_{k-1}})_+^2 
                  \chi_{ \{w-\psi_{L_{k-1}} > 2^{-k-1}\}}  \\  
                                & \displaystyle \leq 2^{(k+1)(p^{*}-2)} \int_{Q_{k-1}} (w-\psi_{L_{k-1}})_+^{p^{*}} \leq 2^{(k+1)(p^{*}-2)} C_N^{p^{*}} U_{k-1}^{p^{*}/2}\\
\displaystyle \int_{Q_{k-1}} (w-\psi_{L_k})_+ & \displaystyle \leq \int_{Q_{k-1}}(w-\psi_{L_{k-1}})_+ 
                 \chi_{ \{w-\psi_{L_{k-1}} > 2^{-k-1}\}}  \\
                               &\displaystyle \leq 2^{(k+1)(p^{*}-1)} \int_{Q_{k-1}}(w-\psi_{L_{k-1}})_+^{p^{*}} \leq 2^{(k+1)(p^{*}-1)} C_{N}^{p^{*}} U_{k-1}^{p^{*}/2} \\
\displaystyle      \int_{Q_{k-1}} \chi_{ \{w-\psi_{L_k} >0\} } & \displaystyle \leq (2^{k+1})^{p^{*}}  \int_{Q_{k-1}}(w-\psi_{L_{k-1}})_+^{p^{*}} \leq 2^{(k+1)p^{*}} C_N^{p^{*}} U_{k-1}^{p^{*}/2}.                                
\end{cases}
\end{equation*}
Combining the above three inequalities with~\eqref{Eq:formulation:energy7} we get
\begin{multline}\label{Eq:nonlinearrecurrence}
U_{k} \leq 2^{k} \left(\frac{1 + \Lambda_{2}}{\Lambda_{1}} \right) \\ \times C_{s,\upgamma, \Lambda, N} \left\lbrace  2^{(k+1)(p^{*}-2)} C_N^{p^{*}} U_{k-1}^{p^{*}/2} + 2^{(k+1)(p^{*}-1)} C_N^{p^{*}} U_{k-1}^{p^{*}/2} + 2^{(k+1)(p^{*}-2)} C_N^{p^{*}} U_{k-1}^{p^{*}/2}   \right\rbrace  \\
\leq (2^{k+1})^{p^{*}} \left(\frac{1 + \Lambda_{2}}{\Lambda_{1}} \right)\widebar{C}_{s,\upgamma, \Lambda, N} U_{k-1}^{p^{*}/2}, 
\end{multline}
for all $k \geq 0$, and for some constant $\widebar{C}$ depending on $s,\upgamma, \Lambda, N$. \medskip

{}From the nonlinear recurrence relation obtained in \eqref{Eq:nonlinearrecurrence}, there exist a ---small enough--- constant  $\varepsilon_{0}^{*}$ depending only on $\widebar{C}_{s,\upgamma, \Lambda, N}$, such that if $U_{1} \leq \varepsilon_{0}^{*}$, then it follows that $\lim_{k\to \infty} U_{k} = 0$. By using the classical Sobolev embedding argument and relation~\eqref{Eq:formulation:energy7} we have that
\[
U_{1} \leq C \int_{a}^{T} \int_{\R^N} \left|w-\psi\right|^{2} \ dx \ dt.
\]  
Then $U_{k} \to 0$ implies that 
\[
w \leq \psi + \frac{1}{2} \quad ~~ \text{for}~ t\in [a,T]\times \R^N.
\]
Next we state a corollary of Lemma~\eqref{lem:firstdeGiorgi}, which shows that any solutions are indeed bounded for $t>0$ and for the constant 
\begin{equation}\label{Eq:relation:delta}
\delta_{\upgamma}^{*}(\vartheta) = \frac{\inf_{1/4 \leq \tau \leq 2} \partial_{t}^{\upgamma}\vartheta(\tau)}{1 + \vartheta(2)- \vartheta(1/4)} >0.
\end{equation}
\begin{corollary}\label{cor:boundedsol}
For $\Norm{f}_{L^{\infty}\left((a,T)\times \R^N \right)} \leq 1$, and for $\delta_{\upgamma}^{*}(\vartheta)>0$ as given in ~\eqref{Eq:relation:delta}, there exists a constant $\nu^{*}\in (0,1)$ and $\varepsilon_{0}>0$, depending only on $N,\Lambda,\upgamma,s$,such that for any solution $w:[a,0]\times\R^N \hookrightarrow \R$ to~\eqref{Eq:main-2} with $\varepsilon < \varepsilon_{0}$, $a \leq -2$ satisfying
\[
\psi(t,x) \leq 1 + \left(|t|^{\upgamma/4}-1\right)_{+} + \left(|x|^{\upgamma/4}-1\right)_{+} \quad \text{if}~~ t\leq 0,
\]
and 
\[
\left|\left\lbrace w>0\right\rbrace \cap \left([-2,0] \times B_{2} \right) \right| \leq \nu^{*},
\]
we have 
\[
w(t,x) \leq \frac{1}{2} \quad \text{for} \quad (t,x) \in [-1,0]\times B_{1}.
\]
\end{corollary}
\begin{proof}
Fixing $(t_{0},x_{0})$ in $[-1,0] \times B_{1}$, for some large value $R$ depending on $s,\upgamma, N$, we introduce the rescaled function $w_{R}$ defined on $\displaystyle \left[R^{s/\upgamma}(a-\tau_{0}), R^{s/\upgamma}-t_{0} \right]\times \R^N $ by 
\[
w_{R}(\tau,y):= w\left(t_{0} + \frac{\tau}{R^{s/\upgamma}}, x_{0} + \frac{y}{R} \right).
\]
This rescaled function $w_{R}$ solves the equation
\[
\partial_{t}^{\upgamma} \vartheta(w_{R}) + \K_{R}w_{R} = f_{R},
\]
with the rescaled kernel associated to $\K$ as
\[
\K_{R}:= \frac{1}{R^{s+N}} \K\left(t_{0} + \frac{\tau}{R^{s/\upgamma}}, x_{0} + \frac{y}{R} \right).
\]
The right hand side of the equation is
\[
f_{R}:=f\left(t_{0} + \frac{\tau}{R^{s/\upgamma}}, x_{0} + \frac{y}{R} \right),
\]
with $\Norm{f_{R}} \leq 1$. Now we study the condition for which for $R$ large enough so that $w_R(\tau,y) \leq \psi(\tau,y)$ for any $(\tau,y) \notin [-R^{s/ \upgamma},0]\times B_R$ we have
\[
\int_{R^{\frac{s}{\upgamma}}(a-t_0)}^{0} \int_{\R^N} [w_R(\tau,y)- \psi(\tau,y)]_+^{2} \leq \varepsilon_{0}^{*}.
\]
We have defined $\psi^{*} := \left(|t|^{\upgamma/4}-1\right)_{+} + \left(|x|^{\upgamma/4}-1\right)_{+} + 1$.\medskip

{}From \cite[Corollary 4.2]{Allen1} it follows that 
\begin{multline*}
\int_{R^{\frac{s}{\upgamma}}(a-t_0)}^{0} \int_{\R^N} [w_R(\tau,y)-  \psi(\tau,y)]_+^{2} = \int_{-R^{\frac{s}{\upgamma}}}^{0} \int_{|y| \leq R} [w_R(\tau,y)- \psi(\tau,y)]_+^2   \\
       \leq \int_{-R^{\frac{s}{\upgamma}}}^{0} \int_{|y| \leq R} [w_R(\tau,y)]_+^2 \leq R^{N+ \frac{s}{\upgamma}}\int_{t_{0}-1}^{t_{0}} \int_{\{x_{0}\} + B_{1}} [w_R(\tau,y)]_+^2 \\
       \leq R^{N+ \frac{s}{\upgamma}}\int_{-2}^{0} \int_{B_{2}} [w_R(\tau,y)]_+^2       \leq R^{N+ \frac{s}{\upgamma}}(\psi^{*}(-2,2)+1)^2 \nu^{*}.
\end{multline*}
If we choose $\nu^{*} = R^{-N-\frac{s}{\upgamma}}(1+\psi^{*}(-2,2))^{-2} \varepsilon^{*}_0$, then
\[
w(t_0,x_0) \leq \frac{1}{2}, \quad \text{for}~(t_0,x_0) \in (-1,0) \times B_1.
\]
\end{proof}

\subsection{The second De Giorgi's Lemma}\label{seconddegiorgi}
Our interest is to analyze the quantitative behaviour of the solution $w(t,x)$ of \eqref{Eq:main-vartheta} under the conditions provided in the proof of the first De Giorgi Lemma~\ref{lem:firstdeGiorgi}. In fact, since the kernel of both nonlocal operators (in time and in space) are very oscillating, the goal is to show that a function with a jump discontinuity cannot be in the energy space. So, some conditions to control the nonlinearity are needed in order to avoid that the equation is degenerate or singular. \medskip

To proceed, we need the following functions (see e.g. \cite{Allen1, Allen2, Pablo3, CaffarelliChanVasseur})
\[
F_1(x) := \sup (-1, \inf (0,|x|^2 -9)), \qquad F_2(t):= \sup (-1, \inf (0,|t|^2 -16)),
\]
where $F_1$ is a Lipschitz and compactly supported function in $B_3$ and equal to $-1$ in $B_2$, and $F_2$ is also a Lipschitz and compactly supported function in $[-4,4]$ and is equal to $-1$ in $[-3,3]$. \medskip

We define the two following functions
\[
\psi_{\lambda}(t,x) :=
((|x|- \lambda^{-4/ s})^{s /4}-1)_+   \chi_{\{|x|\geq \lambda^{-4/s} \}} + ((|t|- \lambda^{-4/ \upgamma})^{\upgamma /4}-1)_+   \chi_{\{|t|\geq \lambda^{-4/\upgamma} \}}. 
  \]  
Our Lemma will involve the following sequence of five cutoffs:
\[
\upvarphi_i := 2+ \psi_{\lambda^3}(t,x)  + \lambda^{i}F_1(x) + \lambda^{i}F_2(t), 
\]
for $0 \leq i \leq 4$ and $\lambda < 1/3$.
\begin{lemma}\label{lem:seconddegiorgi}
Assume $C_{1}(\upgamma) \leq \partial_{t}^{\upgamma} \vartheta(\tau) \leq C_{2}(\upgamma)$, for every $\tau \in [1/2,2]$ and let $\nu^{*}$ be a constant defined in Corollary \ref{cor:boundedsol}. For $0<\mu <1/8$ fixed, there exists $\lambda \in (0,1)$, such that for any solution $w:[a,0]\times \R^N \to \R$ to \eqref{Eq:main-vartheta} with $a\leq -4$, and $|f|\leq 1$ satisfying
\[
\begin{cases}
 \displaystyle w(t,x) \leq 2 + \psi_{\lambda^3}(t,x),\quad \text{on} \quad [a,0] \times \R^N, \\
 \displaystyle |\{w<\upvarphi_0\} \cap ((-3,-2) \times B_1)| \geq \mu,
\end{cases}
\]
then we have 
\[
|\{w>\phi_4\} \cap ((-2,0) \times B_2)| \leq \nu^{*}.
\]
\end{lemma}

This Lemma is similar to the Lemma which states the second De Giorgi's approach from \cite{Allen1, Pablo3}, but now with an extra condition on $\partial_{t}^{\upgamma}~\vartheta(\tau) $, given as 
\begin{equation}\label{Eq:condts}
C_{1}(\upgamma)(w-\widebar{\psi})^{3-\upgamma}_{+} \leq \mathcal{B}_{\widebar{\psi}}(w) \leq C_{2}(\upgamma)(w-\widebar{\psi})^{3-\upgamma}_{+},
\end{equation}
with $\widebar{\psi} = 1 + \psi_{\lambda} + \lambda F_{i}$, the truncation function satisfying the improved energy estimate. The proof can be achieved following the same strategy as in \cite{Allen1, CaffarelliChanVasseur} under the condition~\eqref{Eq:condts}. \medskip

Next we are in conditions to prove the regularity results with the support of previous Lemmas and Corollaries.\medskip 

We start by defining the following function 
\[
\psi_{\tau, \lambda}(t,x) := ((|x|- \lambda^{-4/ s})^{\tau}-1)_+   \chi_{\{|x|\geq \lambda^{-4/s} \}} + ((|t|- \lambda^{-4/ \upgamma})^{\tau}-1)_+   \chi_{\{|t|\geq \lambda^{-4/\upgamma} \}}. 
\]  
\begin{lemma}\label{lem:decrease}
Let $\vartheta$ be such that $\delta_{\upgamma}^{*}(\vartheta)>0$, and assume that $C_{1}(\upgamma) \leq \vartheta_{\tau}^{\upgamma}(\tau) \leq C_{2}(\upgamma)$. There exist $\tau_0$ and $\lambda^*$ such that if for any solution to \eqref{Eq:weak} in $[a,0] \times \R^N$ with $|f|\leq \lambda^4$ and $a\leq -4$ with
\[
-2-\psi_{\tau, \lambda}  \leq w \leq 2+\psi_{\tau, \lambda} ,
\]
then we have 
\[
\sup_{[-1,0] \times B_1} w - \inf_{[-1,0]\times B_1} w \leq 4 - \lambda^{*}.
\]
\end{lemma}
\begin{proof}
We fix $\tau>0$ depending on $\lambda,s,\upgamma$ such that
\[
\frac{(|x|^\tau -1)_+}{\lambda^4} \leq (|x|^{s /4}-1)_+  
\quad \text{  and  } \quad \frac{(|t|^\tau -1)_+}{\lambda^4} \leq (|t|^{\upgamma/4}-1)_+,
\]
and we assume that
\[
\left|\{w<\upvarphi_0\} \cap ((-3,-2) \times B_1)\right| > \mu. 
\]
From Lemma \ref{lem:firstdeGiorgi}, we have that 
\[
\left|\{w>\upphi_4\} \cap ((-2,0)\times B_2)\right| \leq \nu^{*}.
\]
One should notice that $-w$ solves \eqref{Eq:main-vartheta} with $\vartheta$ replaced by $(-\vartheta(-\tau))$. From the hypothesis on $\partial_{t}^{\upgamma} \vartheta$ (see \eqref{Eq:properties_beta}) either $w$ or $-w$ will satisfy the hypotheses of Lemma~\ref{lem:seconddegiorgi}. Hence, we assume that is $w$ which solves \eqref{Eq:properties_beta}.\medskip

Next we consider the sequence of rescaled functions 
\[
w_{k+1}(t,x):= \lambda^{-4}(w_{k}-2(1-\lambda^4)), \qquad w_{0} = w.
\]
We have that $w_{k}$ is a weak solution to \eqref{Eq:main-vartheta} with nonlinearity $\vartheta_{k+1}$ given by
\[
\vartheta_{k+1}(s)=\frac{1}{\lambda^{4}}\vartheta_k(\lambda^{4} \tau+ 1-\lambda^{4}),\qquad \vartheta_{0} = \vartheta.
\]
The goal is to prove that for each $k$ we can apply either Lemma~\ref{lem:firstdeGiorgi} or Lemma~\ref{lem:seconddegiorgi}. As shown in \cite{Pablo3}, repeated application of Lemma~\ref{lem:seconddegiorgi} gives that Lemma~\ref{lem:firstdeGiorgi} can be applied after a finite number of steps.\medskip

Since $\lambda$ was chosen so that 
   \[
    \upvarphi_4(t,x)=2-2\lambda^4 \text{  for  } (t,x) \in [-2,0] \times B_2,
   \]
we have that
\[
\left|\{w_{k+1}>0\} \cap ([-2,0]\times B_2)\right| \leq 
\left|\{w>\phi_4\} \cap ((-2,0)\times B_2)\right| \leq \nu^{*}.
\]
On the other hand
$\partial_{\tau}^{\upgamma} \vartheta_{k+1}(\tau)= C_{\upgamma, \lambda}\partial_{\tau}^{\upgamma}\vartheta_k(\lambda^{4}\tau +1-\lambda^{4})$. Since $\lambda^4{\tau}+1-\lambda^{4}\in[1/2,2]$ whenever $\tau\in[1/2,2]$, we have  $C_{1}(\upgamma)\le \partial_{\tau}^{\upgamma} \vartheta_{k}(\tau)\le C_{2}(\upgamma)$ for every $k$.\medskip

Furthermore, since $[1-\lambda^{4},1+\lambda^4]\subset[1/2,2]$, we get $\nu^{*}(\vartheta_k)\ge \bar\delta_{\upgamma}^{*}>0$ for all $k$.
In addition,
\[
w_{k+1} \leq 2+ \frac{\psi_{\tau,\lambda^3}(x,t)}{\lambda^4} \leq 2+\psi_{\lambda^3} 
\leq 2+ \psi^{*}.
\]
With the previous estimate of $ w_{k+1}$, it comes out that $ w_{k+1}$  satisfies  also \eqref{Eq:weak} with right hand side $|f|\leq 1$. 
Then applying Corollary \ref{cor:boundedsol} to $ w_{k+1}$, we get by induction that 
 \[
w(t,x) \leq 2 - \frac{3}{2}\lambda^4  \text{  for  }  (t,x) \in [-1,0] \times B_1.
\]
 \end{proof}
We now end with the proof of the H\"{o}lder regularity result in Theorem~ \ref{theo:regularity} which states that the oscillation of the solution $w$ in the cylinder $\Gamma_{2}$ is reduced in $\Gamma_{1}$ by a factor $\kappa^{*} = (1- \lambda^{*}/4)$. In doing so, we shall use the previous lemmas and corolaries and additionally to that, we follow the idea provided in \cite{Allen1, Pablo3, CaffarelliSoriaVazquez}. \medskip

Indeed, we consider $(t_0,x_0) \in (a , \infty) \times \R^N$ and assume that $(t_0 -a)>4$. We make a translation and a dilation to the origin by considering $\zeta_0 = \inf(1,t_0 /4)^{\upgamma/s}$,
\[
w_0(t,x):= w(t_0 + \zeta_0^{s / \upgamma} t^{s/ \upgamma} , x_0 + \zeta_0 x). 
\]
Then $w_0$ satisfies an equation of the type \eqref{Eq:weak}. Furthermore, the initial time for $w_0$ will be $(a-t_0)\zeta_0^{-s / \upgamma} < -4$.  \medskip
   
Now we consider $\zeta<1$ such that
\[
\frac{1}{1-(\lambda^* /2)} \psi_{\tau,\lambda}(x) \leq \psi_{\tau,\lambda}(x), \text{  for  }  |x|\geq 1/\zeta.
\]
with $\zeta$ a constant depending on $\lambda, \lambda^{*}, \tau$. We define by induction:
\begin{align*}
      w_1(t,x) &= \frac{w_0(t,x)}{\| w_0 \|_{L^{\infty}}+ \lambda^4\| f \|_{L^{\infty}}} ,  \quad (t,x) \in (a,0) \times \R^N, \\
      w_{k+1}(t,x) &= \frac{1}{1-\lambda^* /4} (w_k(\zeta^{s / \upgamma}t,\zeta x)-\overline{w}_k), \quad (t,x) \in (a\zeta^{-s k},0) \times \R^N,
\end{align*}
with
\[
\overline{w}_k := \frac{1}{|B_1|} \int_{-1}^0 \int_{B_1} w_k(t,x) dx dt.
\]
Then, $w_1(t,x)$ is a solution to the equation
\[
\partial_t^{\upgamma} \vartheta_{0}(w_1)+ \mathcal{K}_0 w_1 = f,
\]
$\vartheta_0(\tau)$ satisfies~\eqref{Eq:properties_beta} in its rescaled form, and the operator $\mathcal{K}_0$ is the nonlocal
integral operator associated to the rescaled kernel
\[
\K_{0}(x,y)=\tau_0^{\frac{N+s}{s}}\K(x_0+\tau_{0}^{1/s}x,x_0+\tau_0^{1/s}y),
\]
and 
\[
\partial_t^{\upgamma} \vartheta_k(w_k)+\mathcal{K}_kw_k= \bar{f}, \qquad\vartheta_{k}(\tau)=\dfrac{\vartheta_0\left(\kappa^{*}_{k} \tau +\overline{w}_{k}\right)}{\kappa^{*}_{k}},
\]
where the operator $\K_k$ has associated kernel
\[
\K_{k}(x,y)=R^{-(N+\sigma)(k+1)}\K_0(R^{-(k+1)}x,R^{-(k+1)}y),
\]
for a given $R>1$.

On the other hand, $w_k$ satisfies the hypothesis of Lemma \ref{lem:decrease} for any $k$. The end of the proof, follows from the decreasing of oscillation of certain mean value of $w_{k}$ in the cylinder $Q_{k} = \Gamma_{R^{-k}}$, the modulus of continuity and the iteratively control of degeneracy and non degeneracy points from \cite{Pablo3, CaffarelliSoriaVazquez} in order to conclude that
\[
C(1-\lambda^* /4)^k \geq \sup_{\rho} w 
  ~-~  \inf_{\rho} w,
\]
with $\varpi:=t_0 +(-\zeta^{s/ \upgamma},0) \times (x_0 + B_{\zeta^k})$.
Thus $w$ is $C^{\beta}$ with 
\begin{equation}\label{e:beta}
\beta = \frac{\log (1-\lambda^* /4)}{\log \zeta^{s/ \upgamma}}.
\end{equation}
\end{proof}


\section*{Appendix A. Proof of Lemma~\ref{lem:contraction-estimate}}\label{Appendix1}

\begin{proof}
We split the proof in two steps:\medskip

{\sc STEP 1}. We start with the computation of the fractional derivative of $\vartheta \eta$.
\[
\partial_{\varepsilon}^{\upgamma}\left(\eta \vartheta\right)(\varepsilon j)=\upgamma \varepsilon \sum _{i<j} \frac{\left[\vartheta(\varepsilon j)\eta(\varepsilon j)- \vartheta(\varepsilon i)\eta(\varepsilon i)\right]}{\left( \varepsilon(j-i)  \right)^{1+\upgamma}}.
\]
Also, we have
\[
\vartheta(\varepsilon j)\eta(\varepsilon j)-\vartheta(\varepsilon i)\eta(\varepsilon i) = \eta(\varepsilon j) \left(\vartheta(\varepsilon j)-\vartheta(\varepsilon i) \right) + \vartheta(\varepsilon i) \left(\eta(\varepsilon j) -\eta(\varepsilon i) \right).
\]
On the other hand, since $\eta \in C^{\infty}$, we set 
\[
\eta_{1}(t,s):=\frac{\eta(t)-\eta(s)}{(t-s)^{1+\upgamma}} ,
\]
which is also smooth. Then we obtain
\begin{multline}\label{Eq:computionA}
\partial_{\varepsilon}^{\upgamma}\left(\eta \vartheta\right)(\varepsilon j)=\upgamma \varepsilon \sum _{i<j} \frac{\eta(\varepsilon j)\left(\vartheta(\varepsilon j)- \vartheta(\varepsilon i)\right)}{\left( \varepsilon(j-i)  \right)^{1+\upgamma}} + \upgamma \varepsilon \sum _{i<j} \frac{  \eta(\varepsilon i)  \left(\eta(\varepsilon j)- \eta(\varepsilon i)\right)}{\left( \varepsilon(j-i)  \right)^{1+\upgamma}} \\
 = \eta(\varepsilon j)\partial_{\varepsilon}^{\upgamma} \vartheta(\varepsilon j) + \varepsilon \sum _{i<j} \frac{\tilde{\eta}(\varepsilon i)\vartheta(\varepsilon i)}{\left( \varepsilon(j-i)  \right)^{\upgamma}}.
\end{multline}
Plugging \eqref{Eq:computionA} into \eqref{Eq:main-2}, we get that $\vartheta \eta$ satisfies \eqref{Eq:main-2} with right hand side
\[
\tilde{f} = f - \varepsilon \sum_{i<j} \frac{\eta_{1}(\varepsilon j, \varepsilon i) \vartheta (\varepsilon i)}{\left( \varepsilon(j-i)  \right)^{\upgamma}}.
\]
Since $\eta_{1} \in C^{\beta}$, $w \in C^{\beta}$ with  $\left(\mathrm{Tr}(\vartheta)\right)^{1/m} = w \in L^{1}(Q) \cap L^{\infty}(Q)$ for $t>0$, and $w$ is bounded, it follows that  the following inequality holds
\begin{multline*}
\left|\tilde{f} \left(\varepsilon(j+h)\right) - \tilde{f} \left(\varepsilon i\right)\right|\\
 = \left|f \left(\varepsilon(j+h)\right) - f \left(\varepsilon i\right) - \frac{\e}{h^{\beta}} \sum_{0<i<j+h}\frac{\tilde{\eta}(\e i) \vartheta(\e i)}{(\e(j+h-i))^{\upgamma}} 
     + \frac{\e}{h^{\beta}} \sum_{0<i<j}\frac{\tilde{\eta}(\e i) \vartheta(\e i)}{(\e(j-i))^{\upgamma}} \right| \\
      \leq \left|f \left(\varepsilon(j+h)\right) - f \left(\varepsilon i\right) \right| + \left|\frac{\e}{h^{\beta}} \sum_{0<i<j+h}\frac{\tilde{\eta}(\e i) \vartheta(\e i)}{(\e(j+h-i))^{\upgamma}} - \frac{\e}{h^{\beta}} \sum_{0<i<j}\frac{\tilde{\eta}(\e i) \vartheta(\e i)}{(\e(j-i))^{\upgamma}} \right|  .
\end{multline*}
As $\varepsilon i \to 0$, we have $w(0,x) \sim \vartheta^{1/m}(0,x)$ and $\tilde{f}(t)$ coincides with $\displaystyle \lim_{\varepsilon h \to 0}\tilde{f}\left(\varepsilon(j+h)\right)$. Hence
\[
\left|\tilde{f} (t) - \tilde{f}(s)\right| \leq  \left|f(t)-f(s)\right| + C(\upgamma) \left| \vartheta^{1/m}(0,x) \right|\left|t-s\right|^{\beta}.
\]
Since $f\in C^{\beta}$, then $\left|f(t)-f(s)\right| \leq C\left|t-s\right|^{\beta}$ and finally
\[
\left|\tilde{f} (t) - \tilde{f}(s)\right| \leq \widetilde{C} \left|t-s\right|^{\beta},
\]
with $\widetilde{C} = C(\left| \vartheta^{1/m}(0,x) \right|, \upgamma)$.\medskip

{\sc STEP 2}. Next we establish contractivity of solution to problem \eqref{Eq:main-2}. As we already mentioned, for simplicity we assume the kernel $\K$ to be time-independent. By using the Crandall-Liggett theorem \cite{Crandall-Liggett} for the implicit discretization of nonlinear evolution problem, we consider the difference quotient
\[
u:=\frac{\eta\vartheta\left(\varepsilon(j+h)\right) - \vartheta (\varepsilon j)}{\left(\varepsilon h\right)^{\beta}}.
\]
It comes out that from \cite[Remark 6.3]{Allen1} and \cite[Theorem 3.2]{Pablo1} once can show that $u\in L^{\infty}$. From the first step of the proof, $u$ satisfies \eqref{Eq:main-2} with $L^{\infty}$ at the right hand side. Thus, $u$ satisfies the estimate \eqref{Eq:contraction3} with modulus $2\rho(\varepsilon)$  \cite{CaffarelliCabre1997}.
\end{proof}

\begin{remark}\label{rem:estimateweak}
One can show that the pair solution $(u,\vartheta)$ to \eqref{Eq:weak} are strong in time-solution. In doing so, some regularity of $\vartheta^{1/m}(0,x)$ is required in order to ensure that $u$ is continuous up to the initial time. This type of technique was used in \cite{Allen1}. Proceeding similarly one can show that $u \in C^{0,\beta}\left([0,T]\times \R^N\right)$, which implies that $\vartheta^{1/m}(0,x) \in C^{0,s}(\R^N)$.
\end{remark}

\section*{Appendix B. Some details for the proof of existence of solutions}\label{Appendix2}
We now show that $\mathcal{\bar{H}}(w,\vartheta,\upvarphi) \to 0$ as $\varepsilon \to 0$. First of all,
\begin{multline*}
\lim_{\e \to 0}\int_{\R^N} \int_0^T \int_0^t \frac{(\vartheta_{\e}(t)-\vartheta_{\e}(s)) (\upvarphi(t)-\upvarphi(s))}{(t-s)^{1+\upgamma}} \ ds \ dt \\
 = \lim_{\e \to 0} \mathop{\sum \sum}_{0\leq i<j\leq k} \int_{\e (j-1)}^{\e j}        \int_{\e (i-1)}^{\e i} \frac{(\vartheta_{\e}(\e j)-\vartheta_{\e}(\e i)) (\upvarphi(\e j)-\upvarphi(\e i))}{(\e(j-i))^{1+\upgamma}}. 
\end{multline*}
We write $\upvarphi(t)= \upvarphi(t)-\upvarphi_{\e}(t)+\upvarphi_{\e}(t)$ and subtract the above two terms. 
   Since $\upvarphi_{\e}(t) \to \upvarphi(t)$ and $\vartheta_{\e} \rightharpoonup \vartheta$ in $H^{\upgamma/2} \times \R^N$ we have that
\[
\lim_{\e \to 0} \left| \int_{\R^N} \int_0^T \int_0^t \frac{(\vartheta_{\e}(t)-\vartheta_{\e}(s))\left[(\upvarphi(t)-\upvarphi(s))
        -(\upvarphi_{\e}(t)- \upvarphi_{\e}(s))\right]}{(t-s)^{1+\upgamma}} \ ds \ dt \right| \to 0.
    \]
Next we must show that
\begin{equation}\label{e:1/4bound}
 \lim_{\e \to 0} 
      \mathop{\sum \sum}_{0\leq i<j\leq k} (\vartheta_{\e}(\e j)-\vartheta_{\e}(\e i))
      (\upvarphi(\e j)-\upvarphi(\e i)) 
       \int_{\e (j-1)}^{\e j} \int_{\e (i-1)}^{\e i} 
         \frac{1}{(t-s)^{1+\upgamma}}-\frac{1}{(\e(j-i))^{1+\upgamma}}=0.
\end{equation}
We break up the integral over two sets $(t-s)\leq \e^{1/4}$ and $(t-s)> \e^{1/4}$.
As a consequence of the strong and weak convergence of $\upvarphi_{\e}$ and $\vartheta_{\e}$ we have that 
\[
\begin{aligned}
0 &= \lim_{\e \to 0} C \int_{\R^N} \mathop{\int \int}_{t-s\leq \e^{1/4}}
\frac{\left|(\vartheta_{\e}(t)-\vartheta_{\e}(s))(\upvarphi_{\e}(t)-\upvarphi_{\e}(s))\right|}{(t-s)^{1+\upgamma}} \\
&\geq \lim_{\e \to 0} \int_{\R^N} \mathop{\sum \sum}_{\e (j-i) \leq \e^{1/4}}
\left|(\vartheta_{\e}(\e j)-\vartheta_{\e}(\e i))(\upvarphi(\e j)-\upvarphi(\e i))\right| \int_{\e (j-1)}^{\e j} \int_{\e (i-1)}^{\e i} 
\frac{1}{(\e(j-i))^{1+\upgamma}}.
\end{aligned}
\]
Now for $t-s>\e^{1/4}$, and $\e(i-1)\leq s \leq \e i$ and $\e (j-1)\leq t \leq \e j$
\begin{multline*}
\left| (t-s)^{-(1+\upgamma)} - (\e(j-i))^{-(1+\upgamma)}\right|
      \leq (\e^{1/4}-\e)^{-(1+\upgamma)} - (\e^{1/4})^{-(1+\upgamma)}\\
     \leq (1+\upgamma)(\e^{1/4}-\e)^{-(2+\upgamma)}\e 
     \leq (1+\upgamma)(\e^{1/4}/2)^{-(2+\upgamma)}\e 
     \leq C(\upgamma)\e^{(2-\upgamma)/4}. 
     \end{multline*}
and so \eqref{e:1/4bound} holds. \medskip
    
We now consider the parts in time
\[
\left| \int_{\R^N} \int_0^{T} \vartheta_{\e}(t)\upvarphi(t)\left[(T-t)^{-\upgamma}+ (t)^{-\upgamma} \right] \ dt \ dx - 
\int_{\R^N} \e \mathop{\sum \sum}_{0\leq i<j\leq k} \frac{\vartheta_{\e}(\e j)\upvarphi(\e j - \vartheta_{\e}(\e i)\upvarphi(\e i))}{(\e(j-i))^{1+\upgamma}}\right|.
    \]
By staying $\e^{1/2}$ away from $0$ and $T$ we have 
    \[
   \begin{aligned}
      &\left| \e^2 \sum_{T-\sqrt{\e}} \sum_{0 \leq i < 2j-k}
        \frac{\vartheta_{\e}(\e j)\upvarphi(\e j)}{(\e(j-i))^{1+\upgamma}}\right| 
        \leq C \int_{T-\e^{1/2}}^{T} \vartheta_{\e}(t) \upvarphi(t)\left[(T-t)^{-\upgamma} + t^{-\upgamma} \right]\to 0, \\
      & \left| \e^2 \sum_{i\leq \e^{1/2}} \sum_{2i < j \leq k}
        \frac{\vartheta_{\e}(\e j)\upvarphi(\e j)}{(\e(j-i))^{1+\upgamma}}\right| 
        \leq C \int_{0}^{\e^{1/2}} \vartheta_{\e}(t) \upvarphi(t)\left[(T-t)^{-\upgamma} + t^{-\upgamma} \right]\to 0.
\end{aligned}
\]
Now we change the above sum as we did previously. For $T/2 \leq t < T- \e^{1/2}$ we have the estimate
\begin{multline*}
0  \leq \upgamma\int_0^{2t-T}(t-s)^{-{1+\upgamma}} -\upgamma \e^{1-\upgamma} 
       \sum_{0 \leq i <2j-k}(j-i)^{-(1+\upgamma)}\\
        \leq \upgamma\int_0^{2t-T}(t-s)^{-(1+\upgamma)} \ ds - \upgamma\int_0^{2t-T}(t-(s-\e))^{-(1+\upgamma)} \ ds 
        = (T-t)^{-\upgamma} - (T-t +\e)^{-\upgamma} \\
        \leq \e^{-\upgamma/2} - (\e^{1/2}+\e)^{-\upgamma} 
         \leq \upgamma (\e)^{-(\alpha+1)/2}\e 
       = \upgamma\e^{(1-\upgamma)/2} \to 0.
\end{multline*}
Hence,
\begin{multline*}
\upgamma \mathop{\sum \sum}_{0\leq i<j\leq k} \frac{\vartheta_{\e}(\e j)\upvarphi(\e j - \vartheta_{\e}(\e i)\upvarphi(\e i))}{(\e(j-i))^{1+\upgamma}} \\
= \e \sum_{\e^{1/2} \leq \e j \leq T-\e^{1/2}}
\vartheta_{\e}(\e j)\upvarphi(\e j) \left[(T- \e j)^{-\upgamma} + (\e j)^{-\upgamma} \right] + o(1),
\end{multline*}
and therefore
\begin{equation*}
\lim_{\e \to 0}\e \sum_{\e^{1/2} \leq \e j \leq T-\e^{1/2}}
\vartheta_{\e}(\e j)\upvarphi(\e j) \left[(T- \e j)^{-\upgamma} + (\e j)^{-\upgamma} \right] 
\to \int_{0}^{T} \vartheta(t)\upvarphi(t)\left[(T- t)^{-\upgamma} + (t)^{-\upgamma} \right]. 
\end{equation*}

The only remaining part in time is 
\[
\int_{\R^N} \int_0^T \vartheta_{\e}(t) \partial_{\e}^{\upgamma} \upvarphi(t) \ dt \ dx - \int_{\R^N}  \e \sum_{0<j\leq k} \vartheta_{\e}(\e j) \partial_{\e}^{\upgamma} \upvarphi(\e j).   
\]
Since $\upvarphi$ is smooth we have that    $\vartheta_{\e} \rightharpoonup \vartheta$ in $H^{\upgamma/2}(0,T) \times \R^N$, by using \cite{Allen1} and Lemma~\ref{lem:contraction-estimate}. Hence in combination with the previous techniques this term also goes to zero.\medskip

The remaining term is the spatial term given by
\[
\lim_{\e \to 0} \int_a^T \mathcal{E}_{\e}(w_{\e},\upvarphi) \ dt 
= \lim_{\e \to 0} \e \sum_{0<j\leq k} \mathcal{E}(w_{\e}(\e j, x), \upvarphi(\e j,x)),
\]
which can be handled as in \cite{Allen1}.  \medskip
  
We notice that it could also be sufficient to show \eqref{Eq:main-vartheta} for kernels $\K(t,x,y)$ which are smooth and $0 \leq \K(t,x,y)\leq M$, with $|x-y| < M^{-1}$ for $M$ large. If $0 \leq \K(t,x,y) \leq M$, then we can find a sequence of solutions $w_k$ to \eqref{Eq:main-vartheta} with 
smooth kernels $\K_\e$. Then by convergence in the appropriate spaces, $w_\e \to w_M$ a solution to \eqref{Eq:main-vartheta} with bounded kernel $0\leq \K \leq M$. Then letting $M$ go to infinity one can obtain the desired  result.\medskip
 
Indeed, let us assume that the kernel $\K(t,x,y)$ is bounded and smooth. We define
\[
V(t,x,y):=\int_{a}^T \int_{\R^N}\int_{\R^N}\K(t,x,y)(w(t,x)-w(t,y))(\upvarphi(t,x)-\upvarphi(t,y))\ dx \ dy \ dt,
\]
and
\[
V_{\e}(t,x,y):=\K_{\e}(w_{\e}(t,x)-w_{\e}(t,y))
(\upvarphi_{\e}(t,x)-\upvarphi_{\e}(t,y)) \ dx \ dy \ dt.
\]
Then it follows that
\[
\begin{aligned}
&\left| V(t,x,y) - V_{\e}(t,x,y) \right| \\
&\leq \left|\int_{a}^T \int_{\R^N}\int_{\R^N} \K\left\{ (w(t,x)-w(t,y))  
      - (w_{\e}(t,x)-w_{\e}(t,y)) \right\} (\upvarphi(t,x)-\upvarphi(t,y)) \ dx \ dy \ dt \right| \\ 
     & \ +  \bigg| \int_{a}^T \int_{\R^N}\int_{\R^N}(w_{\e}(t,x)-w_{\e}(t,y))  \\
     & \times 
\left\{\K(\upvarphi(t,x)-\upvarphi(t,y)) - \K_{\e}(t,x,y)(\upvarphi_{\e}(t,x)-\upvarphi_{\e}(t,y)) \right\} \ dx \ dy \ dt \bigg|. \\
\end{aligned}
\]
Since $w_{\e} \to w$ in $(a,T) \times H^{\frac{s}{2}}(\R^N)$ the first term converges to zero. Also, this same convergence coupled with $\K$ and $\upvarphi$ being both smooth and bounded implies the second term goes to zero as well. \medskip

We have thus obtained a weak solutions $(w, \vartheta)$ to the problem \eqref{Eq:main-vartheta}. The uniqueness can be handled similarly as in \cite{Allen1} under the condition that $\vartheta$ satisfies \eqref{Eq:properties_beta} and Lemma~\ref{lem:contraction-estimate}.


\section*{Acknowledgements}

This work has been partially supported by the Agencia Estatal de Investigaci\'on (AEI) of Spain under grant MTM2016--75140--P, co-financed by the European Community fund FEDER, and Xunta de Galicia, grants GRC 2015--004 and R 2016/022.


\end{document}